\documentclass[11pt,reqno]{article}
\usepackage{amsmath,amsfonts,amsthm,amssymb}
\usepackage{graphics}
\usepackage[all]{xy}
\usepackage{color}
\usepackage{multicol}
\usepackage{amscd}

  \theoremstyle{plain}
\newtheorem{theorem}{Theorem}[section]
\newtheorem{proposition}[theorem]{Proposition}

\newtheorem{corollary}[theorem]{Corollary}

\theoremstyle{definition}
\newtheorem{definition}[theorem]{Definition}
\newtheorem{example}[theorem]{Example}

\xyoption{knot} \UseComputerModernTips \knottips{FF}
\def\Bracket#1{\mathord{\Big\langle\,~ \raise8pt\xybox{0;/r1.3pc/:#1}\,
~\Big\rangle}}

\begin{document}

\title{Computations of quandle cocyle invariants of surface-links using marked graph diagrams}

\author{\it\small Seiichi Kamada, Jieon Kim and Sang Youl Lee
\\{}\\{\it\small Department of Mathematics,   Osaka City University,}\\
{\it\small Sugimoto, Sumiyoshi-ku, Osaka 558-8585, Japan}\\{\it\small E-mail: skamada@sci.osaka-cu.ac.jp,}
\\{\it\small Department of Mathematics, Graduate School of Natural Sciences,}\\{\it\small Pusan National University,}\\
{\it\small Busan 609-735, Republic of Korea}\\{\it\small E-mail: jieonkim@pusan.ac.kr}
\\{\it\small Department of Mathematics, Pusan National University,}\\
{\it\small Busan 609-735, Republic of Korea}\\{\it\small E-mail: sangyoul@pusan.ac.kr}}

\maketitle

\begin{abstract}
 By using the cohomology theory of quandles, quandle cocycle invariants and shadow quandle cocycle invariants
are defined for oriented links and surface-links via broken surface diagrams.  By using symmetric quandles, symmetric quandle cocycle invariants are also defined for unoriented links and surface-links via broken surface diagrams.
A marked graph diagram is a link diagram possibly with $4$-valent vertices equipped with markers. S. J. Lomonaco, Jr. and K. Yoshikawa  introduced a method of describing surface-links by using marked graph diagrams. In this paper, we give interpretations of these quandle cocycle invariants in terms of marked graph diagrams, and introduce a method of computing them from marked graph diagrams.
\end{abstract}


\section{Introduction}\label{sect-intr}

A {\it surface-link}  is a closed 2-manifold smoothly (or piecewise linearly and locally flatly) embedded in the Euclidian $4$-space $\mathbb R^4$.  
Two surface-links $\mathcal L$ and $\mathcal L'$ are said to be {\it equivalent} if there exists an orientation preserving homeomorphism $h:\mathbb R^4\rightarrow \mathbb R^4$ such that $h(\mathcal L)=\mathcal L'.$ When $\mathcal L$ and $\mathcal L'$ are oriented, it is assumed that $h|_{\mathcal L}:\mathcal L\rightarrow \mathcal L'$ is also an orientation preserving homeomorphism.

A {\it broken surface diagram} of a surface-link is a projection image in $\mathbb R^3$ with over/under sheet information at each double point curve. It is known that two broken surface diagrams present equivalent surface-links if and only if they are related by a finite sequence of Roseman moves (cf. \cite{Ro}).  

A marked graph diagram is a link diagram possibly with $4$-valent vertices equipped with markers. S. J. Lomonaco, Jr. \cite{Lo} and K. Yoshikawa \cite{Yo} introduced a method of describing surface-links by using marked graph diagrams. Yoshikawa introduced local moves on marked graph diagrams, which are so-called {\it Yoshikawa moves}. Two marked graph diagrams present equivalent surface-links if and only if they are related by a finite sequence of Yoshikawa moves (\cite{KK,KJL2,Sw}). So one can use marked graph diagrams for studying surface-links and their invariants (cf. \cite{As,JKL,JKaL,KJL,KJL2,Le2,Le4,Le1,So}).

A {\it quandle} is a set $X$ with a binary operation $\ast:X\times X\rightarrow X$ satisfying certain conditions derived from Reidemeister moves for classical link diagrams (\cite{Jo,Ma}). By using the cohomology theory of quandles  (\cite{CJKLS, CJKS01b, FeRoSa1,FeRoSa2,Fl,Gr}), {\it quandle cocycle invariants} and {\it shadow quandle cocycle invariants} 
are defined for oriented links and surface-links via broken surface  diagrams (\cite{CKS, CJKLS, CJKS01a}).  On the other hand, by using symmetric quandles, {\it symmetric quandle cocycle invariants} are also defined for unoriented links and surface-links via broken surface diagram (\cite{Ka2,KO}). These invariants for surface-links are defined as state-sums over all quandle colorings of sheets and corresponding Boltzman weights that are evaluations of a cocycle at triple points in broken surface diagrams.

The aim of this paper is to interpret of these quandle cocycle invariants in terms of marked graph diagrams, and introduce a method of computing the quandle cocycle invariants from marked graph diagrams.

This paper is organized as follows: In Section \ref{sect-mgd}, we prepare some preliminaries about broken surface diagrams and marked graph diagrams. Section \ref{sect-qcoc} contains a review of quandle cocycle invariants of oriented surface-links. In Section \ref{sect-qcocm}, we describe quandle cocycle invariants via marked graph diagrams and give a method of  computing the quandle cocycle invariants from marked graph diagrams. Section \ref{sect-sqcoc} contains shadow colorings and shadow quandle cocycle invariants of oriented surface-links. In Section \ref{sect-sqcocm}, we describe how to compute shadow quandle cocycle invariants from marked graph diagrams. In Section \ref{sect-syqcoc}, we recall symmetric quandles and symmetric quandle cocycle invariants of unoriented surface-links. Section \ref{sect-syqcocm} is devoted to giving a method of computing symmetric quandle cocycle invariants from marked graph diagrams.


\section{Marked graph diagrams of surface-links}
\label{sect-mgd}

In this section, we recall broken surface diagrams and marked graph diagrams presenting surface-links.

Let $\mathcal L$ be a surface-link.  By deforming $\mathcal L$ by an ambient isotopy of $\mathbb{R}^4$ if necessary, we may assume that the restriction map $q|_{\mathcal L} : {\mathcal L} \to  \mathbb{R}^3$ is a general position map, where $q:\mathbb{R}^4\rightarrow\mathbb{R}^3$ denotes the projection 
$(x, y, z, t) \mapsto (x, y, t)$.  Along the double point curves, one of the sheets (called the {\it over-sheet}) lies above the other ({\it under-sheet}) with respect to the $z$-coordinate. The under-sheets are coherently broken in the projection. The union $\mathcal B$ of such broken surfaces is called a {\it broken surface diagram} of $\mathcal L$.  
 When $\mathcal L$ is an oriented surface-link, we assume that the sheets of are co-oriented such that the pair (orientation, co-orientation) matches the given (right-handed) orientation of $\mathbb{R}^3.$ In \cite{Ro}, D. Roseman introduced seven moves of broken surface diagrams, called {\it Roseman moves}.  Two surface-links are equivalent if and only if their broken surface diagrams are related by a finite sequence of Roseman moves. For more details, see \cite{CS,Ro}.

\bigskip

A {\it marked graph} is a spatial graph $G$ in $\mathbb R^3$ which satisfies the following:
\begin{itemize}
  \item [(1)] $G$ is a finite regular graph with $4$-valent vertices, say $v_1, v_2, . . . , v_n$.
  \item [(2)] Each $v_i$ is a rigid vertex; that is, we fix a rectangular neighborhood $N_i$ homeomorphic to $\{(x, y)|-1 \leq x, y \leq 1\},$ where $v_i$ corresponds to the origin and the edges incident to $v_i$ are represented by $x^2 = y^2$.
  \item [(3)] Each $v_i$ has a {\it marker}, which is the interval on $N_i$ given by  $\{(x, 0)|-1 \leq x \leq 1\}$.
\end{itemize}

Two marked graphs are said to be {\it equivalent} if they are ambient isotopic in $\mathbb R^3$ with keeping the rectangular neighborhoods and markers.

An {\it orientation} of a marked graph $G$ is a choice of an orientation for each edge of $G$ such that every vertex in $G$ looks like 
\xy (-5,5);(5,-5) **@{-}, 
(5,5);(-5,-5) **@{-}, 
(3,3.2)*{\llcorner}, 
(-3,-3.4)*{\urcorner}, 
(-2.5,2)*{\ulcorner},
(2.5,-2.4)*{\lrcorner}, 
(3,-0.2);(-3,-0.2) **@{-},
(3,0);(-3,0) **@{-}, 
(3,0.2);(-3,0.2) **@{-}, 
\endxy or 
\xy (-5,5);(5,-5) **@{-}, 
(5,5);(-5,-5) **@{-},  
(2.5,2)*{\urcorner}, 
(-2.5,-2.2)*{\llcorner}, 
(-3.2,3)*{\lrcorner},
(3,-3.4)*{\ulcorner},
(3,-0.2);(-3,-0.2) **@{-},
(3,0);(-3,0) **@{-}, 
(3,0.2);(-3,0.2) **@{-}, 
\endxy.  
A marked graph $G$ is said to be 
{\it orientable} if it admits an orientation. Otherwise, it is said to be {\it non-orientable}. 
Figure~\ref{fig-nori-mg} shows an oriented marked graph and a non-orientable marked graph.  
Marked graphs can be described by  diagrams on $\mathbb R^2$ with some $4$-valent vertices equipped with markers. 

\begin{figure}[ht]
\begin{center}
\resizebox{0.5\textwidth}{!}{%
  \includegraphics{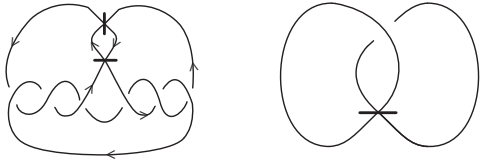}}
\caption{Marked graphs}\label{fig-nori-mg}
\end{center}
\end{figure}

A surface-link $\mathcal L$ in $\mathbb R^4 =\mathbb R^3 \times \mathbb R$ can be described in terms of its {\it cross-sections} $\mathcal L_t=\mathcal L \cap \mathbb R^3\times \{t\}$,  $ t \in \mathbb R$ (cf. \cite{Fox}).  Let $p:\mathbb R^4 \to \mathbb R$ be the projection given by $p(x_1, x_2, x_3, x_4)=x_4$, and we denote by 
 $p_{\mathcal L} : \mathcal L \to  \mathbb R$ the restriction to ${\mathcal L}$.  
 It  is  known (\cite{KSS,Kaw,Lo}) that any surface-link $\mathcal L$ is equivalent to a surface-link $\mathcal L'$, called a {\it hyperbolic splitting} of $\mathcal L$,
such that
the projection $p_{\mathcal L'}: \mathcal L' \to \mathbb R$ satisfies that 
all critical points are non-degenerate, 
all the index 0 critical points (minimal points) are in $\mathbb R^3\times \{-1\}$, 
all the index 1 critical points (saddle points) are in $\mathbb R^3\times \{0\}$, and 
all the index 2 critical points (maximal points) are in $\mathbb R^3\times \{1\}$.

Let $\mathcal L$ be a surface-link and let ${\mathcal L'}$ be a hyperbolic splitting of $\mathcal L.$
The cross-section $\mathcal L'_0=\mathcal L'\cap \mathbb R^3 \times \{0\}$ at $t=0$ is a $4$-valent graph in $\mathbb R^3\times \{0\}$. We give a marker at each $4$-valent vertex (saddle point) that indicates how the saddle point opens up above as illustrated in Figure~\ref{sleesan2:fig1}. The resulting marked graph $G$ is called a {\it marked graph} presenting $\mathcal L$. As usual, $G$ is described by a diagram $\Gamma$ on $\mathbb R^2$ which is a generic projection on $\mathbb R^2$ with over/under crossing information for each double point  such that the restriction to a rectangular neighborhood of each marked vertex is an embedding. Such a diagram is called a {\it marked graph diagram} or a {\it ch-diagram} (cf. \cite{So}) presenting $\mathcal L$. 

When $\mathcal L$ is an oriented surface-link, we assume that 
$\mathcal L'_0$ has the induced orientation as the boundary of the oriented surface $\mathcal L' \cap (\mathbb R^3 \times (-\infty, 0])$. 

\begin{figure}[h]
\begin{center}
\resizebox{0.60\textwidth}{!}{%
  \includegraphics{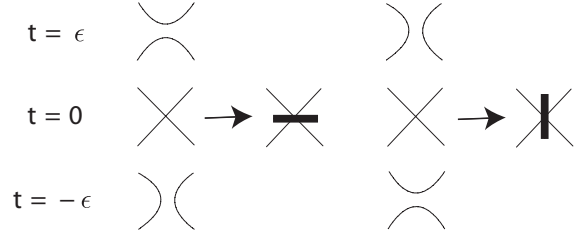} }
\caption{Marking of a vertex} \label{sleesan2:fig1}
\end{center}
\end{figure}

Let $\Gamma$ be a marked graph diagram and $\Gamma_0$ the singular link diagram obtained from $\Gamma$ by removing all markers. Let $V(\Gamma)=\{v_1,v_2,\ldots,v_n\}$ be the set of all vertices of $\Gamma$.  For each $i$ $(i=1,\ldots, n),$ consider four points $v_i^1,v_i^2,v_i^3$, and $v_i^4$ on $\Gamma$ in a neighborhood of $v_i$ as in Figure \ref{fig-ver}. We define
 \begin{align*}
\Gamma_+=\bigl[\Gamma_0\setminus \overset{n}{\underset{i=1}{\cup}}\bigl(\overset{4}{\underset{j=1}{\cup}}|v_i,v_i^j|\bigr)\bigr]\cup\bigl[\overset{n}{\underset{i=1}{\cup}}\bigl(|v_i^1,v_i^2|\cup|v_i^3,v_i^4|\bigr)\bigr],
\end{align*}
which is called the {\it positive resolution} of $\Gamma,$ and
 \begin{align*}
\Gamma_-=\bigl[\Gamma_0\setminus \overset{n}{\underset{i=1}{\cup}}\bigl(\overset{4}{\underset{j=1}{\cup}}|v_i,v_i^j|\bigr)\bigr]\cup\bigl[\overset{n}{\underset{i=1}{\cup}}\bigl(|v_i^1,v_i^3|\cup|v_i^2,v_i^4|\bigr)\bigr],
\end{align*}
the {\it negative resolution} of $\Gamma$, where $|v,w|$ is the line segment connecting $v$ and $w$.  
When both resolutions $\Gamma_-$ and $\Gamma_+$ are diagrams of trivial links, we call 
$\Gamma$  {\it admissible}. 

\begin{figure}[ht]
\begin{center}
\resizebox{0.18\textwidth}{!}{%
  \includegraphics{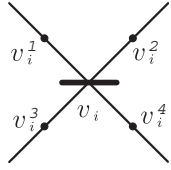} }
\caption{$v_i^1,v_i^2,v_i^3$, and $v_i^4$}\label{fig-ver}
\end{center}
\end{figure}

When $\Gamma$ is admissible, we construct a surface-link as follows (cf.  \cite{Ka2,KSS,Kaw,Yo}).
Let $L_0$ denote  a graph in $\mathbb R^3$ whose diagram is $\Gamma_0$.  Let $w_i^j$ and $w_i$ be points on $L_0$ such that $\pi(w_i^j)=v_i^j,$ $\pi(w_i)=v_i,$ respectively, where $\pi:\mathbb R^3\rightarrow \mathbb R^2$ is the projection $(x,y,z) \mapsto (x,y)$. 
For each $t \in [0,1]$, let $w_i^j(t)$ be the point $ (1-t) w_i + t w_i^j \in \mathbb{R}^3$.  

For each $t \in [0,1]$, let 
$L_t^{+}$ be a link defined by
$$L_t^{+} = \bigl[ L_0 \setminus \overset{n}{\underset{i=1}{\cup}} \bigl( \overset{4}{\underset{j=1}{\cup}} 
| w_i, w_i^j(t)  |   \bigr) \bigr]   \cup \bigl[ \overset{n}{\underset{i=1}{\cup}} \bigl( 
| w_i^1(t), w_i^2(t) | \cup
| w_i^3(t), w_i^4(t) | \bigr)\bigr],$$ 
and for each $t \in [-1,0]$, let $L_-$ be a link defined by 
$$L_{t}^{-} = \bigl[ L_0 \setminus \overset{n}{\underset{i=1}{\cup}} \bigl(\overset{4}{\underset{j=1}{\cup}}
 | w_i, w_i^j(-t) | \bigr)\bigr] 
\cup \bigl[ \overset{n}{\underset{i=1}{\cup}} \bigl( 
| w_i^1(-t), w_i^3(-t) | \cup 
| w_i^2(-t), w_i^4(-t) | \bigr)\bigr]. $$ 
 
Put $L_+ = L_1^{+}$ and $L_- = L_{-1}^{-}$.  Then $L_+$ and $L_-$ have diagrams $\Gamma_+$ and $\Gamma_-$, respectively.  
Let $B_1^{+},\ldots,B_\mu^+$ be mutually disjoint 2-disks in $\mathbb{R}^3$ with 
$\partial(B_1^{+}\cup\cdots\cup B_\mu^+)=L_+$, and let 
$B_1^{-},\ldots,B_\lambda^-$ be mutually disjoint 2-disks  in $\mathbb{R}^3$ with 
$\partial(B_1^{-}\cup\cdots\cup B_\lambda^-)=L_-$.  

Let $F(\Gamma)$ be a surface-link in $\mathbb{R}^4=\mathbb R^3\times \mathbb R$ defined by
\begin{align*}
F(\Gamma) = \/ 
&(B_1^{-}\cup\cdots\cup B_\lambda^-)\times\{-2\} 
\cup L_- \times(-2,-1)\\
&\cup (\cup_{t\in[-1,0)}L_t^-\times\{t\})
\cup L_0 \times\{0\} 
\cup (\cup_{t\in(0,1]}L_t^+\times\{t\})\\
&\cup L_+ \times(1,2) 
\cup (B_1^{+}\cup\cdots\cup B_\mu^+)\times \{2\}. 
\end{align*}

We say that $F(\Gamma)$ is a {\it surface-link associated to $\Gamma$}.  It is uniquely determined from $\Gamma$ up to equivalence (see \cite{KSS}).


A surface-link $\mathcal L$ is said to be {\it presented} by a marked graph diagram $\Gamma$ if $\mathcal L$ is equivalent to the  surface-link $F(\Gamma)$.  Any surface-link can be presented by an admissible marked graph diagram.  Two admissible marked graph diagrams present equivalent surface-links if and only if they are related by a finite sequence of Yoshikawa moves (\cite{KK,KJL2,Sw}).


S. Ashihara  introduced a method of constructing a broken surface diagram of a surface-link from its marked graph diagram \cite{As}. For our later use, we describe here his construction.  
In what follows, by $D\rightarrow D'$ we mean that a link diagram $D'$ is obtained from a  link diagram $D$ by a single Reidemeister move (Figure~\ref{fig-rmove}) or an ambient isotopy of $\mathbb R^2$.
\begin{figure}[ht]
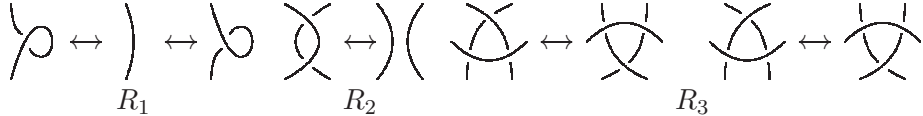

\centerline{
\xy (0,0);(4,7) **\crv{(2,7)}, (4,3);(4,7)
**\crv{(6,3)&(6,7)}, (0,10);(1.5,6) **\crv{(0.2,6.5)},
(2.5,4);(4,3) **\crv{(3,3.1)},
\endxy
\xy (2,5);(6,5) **@{-} ?>*\dir{>} ?<*\dir{<},(0,0)*{},
\endxy
\hskip 0.2cm
\xy (0,0);(0,10) **\crv{(2,5)},
(1,-3) *{R_1},
\endxy
\xy (2,5);(6,5) **@{-} ?>*\dir{>} ?<*\dir{<},(0,0)*{},
\endxy
\hskip 0.2cm
\xy (0,10);(4,3) **\crv{(2,3)}, (4,7);(4,3)
**\crv{(6,7)&(6,3)}, (0,0);(1.5,4) **\crv{(0.2,3.5)},
(2.5,6);(4,7) **\crv{(3,6.9)},
\endxy
\quad
\xy (0,0);(0,10) **\crv{(9,5)}, (6,10);(3.8,8.5) **\crv{(5,9.5)},
(6,0);(3.8,1.5) **\crv{(5,0.5)}, (2.3,2.8);(2.3,7.2)
**\crv{(0.5,5)},
\endxy
\xy (2,5);(6,5) **@{-} ?>*\dir{>} ?<*\dir{<},
(4,-3) *{R_2},(0,0)*{},
\endxy
\xy (0,0);(0,10) **\crv{(4,5)}, 
(6,0);(6,10) **\crv{(2,5)},
\endxy
\quad
\xy (0,5);(10,5) **\crv{(5,0)}, (7.5,4);(2,10) **\crv{(6.7,8)},
(2,0);(2.2,2.2) **\crv{(2,1)}, (8,0);(7.8,2.2) **\crv{(8,1)},
(2.5,4);(4.4,7.7) **\crv{(3,6.5)}, (5.8,8.8);(8,10) **\crv{(6,9)},
\endxy
\xy (2,5);(6,5) **@{-} ?>*\dir{>} ?<*\dir{<},
,(0,0)*{},
\endxy
\hskip 0.2cm
\xy (0,5);(10,5) **\crv{(5,10)}, (2.5,6);(8,0) **\crv{(3.3,2)},
(2,10);(2.2,7.8) **\crv{(2,9)}, (8,10);(7.8,7.8) **\crv{(8,9)},
(7.5,6);(5.6,2.3) **\crv{(7,3.5)}, (4.2,1.2);(2,0) **\crv{(4,1)},
\endxy
\xy (4,-3) *{R_{3}},(0,0)*{},
\endxy
\xy (0,5);(10,5) **\crv{(5,0)}, (2.5,4);(8,10) **\crv{(3.3,8)},
(8,0);(7.8,2.2) **\crv{(8,1)}, (2,0);(2.2,2.2) **\crv{(2,1)},
(7.5,4);(5.6,7.7) **\crv{(7,6.5)}, (4.2,8.8);(2,10) **\crv{(4,9)},
\endxy
\xy (2,5);(6,5) **@{-} ?>*\dir{>} ?<*\dir{<},
,(0,0)*{},
\endxy
\hskip 0.2cm
\xy (0,5);(10,5) **\crv{(5,10)}, (7.5,6);(2,0) **\crv{(6.7,2)},
(8,10);(7.8,7.8) **\crv{(8,9)}, (2,10);(2.2,7.8) **\crv{(2,9)},
(2.5,6);(4.4,2.3) **\crv{(3,3.5)}, (5.8,1.2);(8,0) **\crv{(6,1)},
\endxy
}\caption{Reidemeister moves of type $R_1$, $R_2$ and $R_3$}
\label{fig-rmove}
\end{figure}
  
 Let $\Gamma$ be an admissible marked graph diagram, and let $\Gamma_+$ and $\Gamma_-$ be the positive and the negative resolutions.  
 
Since $\Gamma_+$ is a diagram of a trivial link, there is a sequence of link diagrams from $\Gamma_+$ to a trivial link diagram $O$ related by ambient isotopies of $\mathbb R^2$ and Reidemeister moves: 
$$\Gamma_+=D_1\rightarrow D_2\rightarrow\cdots\rightarrow D_r=O.$$ 
For each $i$ ($i=1, \ldots, r-1$), let $\{f^{(i)}_t\}_{t\in I}$ be a $1$-parameter  family of homeomorphisms from $\mathbb{R}^3$ to $\mathbb{R}^3$ which satisfies 
$$f_0^{(i)}={\rm id}, \quad f_1^{(i)}(L(D_i))=L(D_{i+1}),$$ 
where $L(D_i)$ denotes a link in $\mathbb{R}^3$ with diagram $D_i$ for $i$ ($i=1, \ldots, r$). 
Without loss of generality, we may assume that $L(D_1)= L_+$ and the  
following two conditions are satisfied.  
\begin{itemize}
\item[$\bullet$] When the move $D_i\rightarrow D_{i+1}$ is an ambient isotopy of $\mathbb R^2$, let $\{h^{(i)}_t\}_{t\in I}$ be  an ambient isotopy of $\mathbb R^2$ such that $h_1^{(i)}(D_i)=D_{i+1}.$ Then $f^{(i)}_t$ satisfies $\pi(f_t^{(i)}(L(D_i)))=h_t^{(i)}(\pi(L(D_i)))$ for $t\in I.$ 
\item[$\bullet$] When the move $D_i\rightarrow D_{i+1}$ is a  Reiedemeister move, let $B_{(i)}$ be a disk in $\mathbb R^2$ where the move is applied, and let $M_{(i)}$ be the subset of $B_{(i)}\times I$ $\subset \mathbb R^3$ determined by $\pi(M_{(i)}\cap (B_{(i)}\times \{t\}))=\pi(f_t^{(i)}(L(D_i)))\cap B_{(i)}$ for $t\in I.$ 
Then $M_{(i)}$ is as in Figure~\ref{fig-m1}, \ref{fig-m2}, or \ref{fig-m3}.
\end{itemize} 
 \begin{figure}[ht]
\begin{center}
\resizebox{0.5\textwidth}{!}{%
  \includegraphics{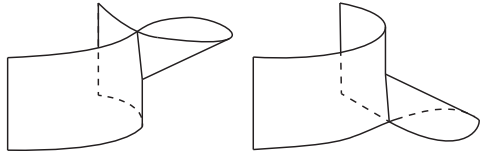} }
\caption{$M_{(i)}$ for $R_1$}
\label{fig-m1}
\end{center}
\end{figure}
 \begin{figure}[ht]
\begin{center}
\resizebox{0.5\textwidth}{!}{%
  \includegraphics{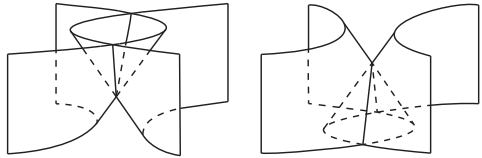} }
\caption{$M_{(i)}$ for $R_2$}
\label{fig-m2}
\end{center}
\end{figure}
 \begin{figure}[ht]
\begin{center}
\resizebox{0.5\textwidth}{!}{%
  \includegraphics{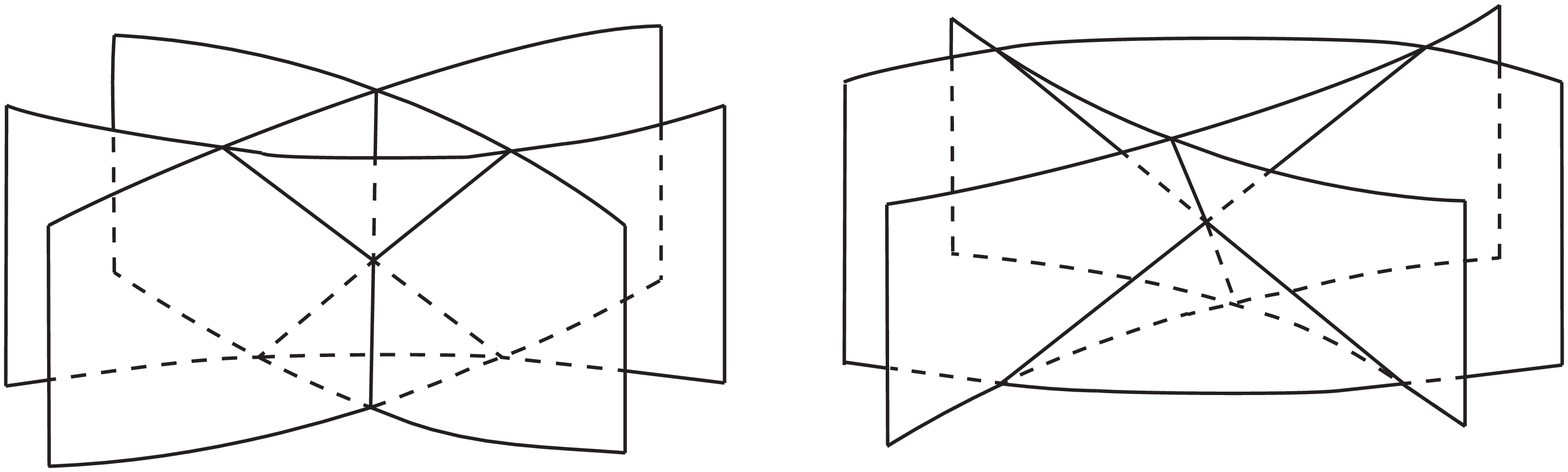} }
\caption{$M_{(i)}$ for $R_3$}
\label{fig-m3}
\end{center}
\end{figure}
Take real numbers $t_1, \ldots, t_r$ with $1 < t_1< \cdots<t_r<2.$ For each $i$ ($i=1, \ldots, r-1$), we define a homeomorphism $F^{(i)}:\mathbb{R}^4 (=\mathbb{R}^3 \times \mathbb{R}) 
\rightarrow\mathbb{R}^4$ by 
\begin{align*}
F^{(i)}(x,t)=\begin{cases}
(x,t) \hskip 2cm (t\leq t_i),\\
(f_{\phi(t)}^{(i)}(x),t) \hskip 1cm (t_i<t< t_{i+1}),\\
(f_{1}^{(i)}(x),t) \hskip 1.15cm (t\geq t_{i+1}),
\end{cases}
\end{align*}
where $\phi(t)=(t-t_i)/(t_{i+1}-t_i).$

Similarly, consider a sequence of link diagrams from $\Gamma_-$ to a trivial link diagram $O',$ related by ambient isotopies of $\mathbb R^2$ and Reidemeister moves:
$$\Gamma_-=D_1' \rightarrow D_2' \rightarrow\cdots\rightarrow D_s'=O'.$$
For each $j$ ($j=1, \ldots,s-1$), let $\{g^{(j)}_t\}_{t\in I}$ be a $1$-parameter family of homeomorphisms from $\mathbb{R}^3$ to $\mathbb{R}^3$ which satisfies $$ g_0^{(j)}={\rm id}, \quad g_1^{(j)}(L(D_j'))=L(D_{j+1}').$$  
Without loss of generality, we may assume that $L(D_1')=L_-$ and 
the following two conditions are satisfied.  
\begin{itemize}
\item[$\bullet$] When the move $D_j'\rightarrow D_{j+1}'$ is an ambient isotopy of $\mathbb R^2$, let $\{{h'_t}^{(j)}\}_{t\in I}$ be  an ambient isotopy of $\mathbb R^2$ such that ${h'_1}^{(j)}(D_j')=D_{j+1}'.$ Then $g^{(j)}_t$ satisfies $\pi(g_t^{(j)}(L(D_j')))={h'_t}^{(j)}(\pi(L(D_j')))$ for $t\in I.$
\item[$\bullet$] When the move $D_j'\rightarrow D_{j+1}'$ is a  Reidemeister move, let $B_{(j)}'$ be a disk in $\mathbb R^2$ where the move is applied, and let $M_{(j)}'$ be the subset of $B_{(j)}'\times I$ $\subset \mathbb R^3$ determined by $\pi(M_{(j)}'\cap (B_{(j)}'\times \{t\}))=\pi(g_t^{(j)}(L(D_j')))\cap B_{(j)}'$ for $t\in I.$ Then $M_{(j)}'$ is as in Figure \ref{fig-m1}, \ref{fig-m2}, or \ref{fig-m3}.
\end{itemize} 
Take real numbers $t_1', \ldots, t_s'$ with $-1 > t_1'> \cdots> t_s' > -2.$ For each $j$ ($j=1, \ldots, s-1$), we define a homeomorphism $G^{(j)}:\mathbb{R}^4\rightarrow\mathbb{R}^4$ by 
\begin{align*}
G^{(j)}(x,t)=\begin{cases}
(x,t) \hskip 2cm (t\geq t_j'),\\
(g_{\psi(t)}^{(j)}(x),t) \hskip 1cm (t_{j+1}'<t< t_j'),\\
(g_{1}^{(j)}(x),t) \hskip 1.15cm (t\leq t_{j+1}'),
\end{cases}
\end{align*}
where $\psi(t)=(t_j'-t)/(t_j'-t_{j+1}').$

Let $F'=G^{(s-1)}\circ G^{(s-2)}\circ\cdots\circ G^{(1)}\circ F^{(r-1)}\circ F^{(r-2)}\circ\cdots\circ F^{(1)}(F(\Gamma))$.
Then $F'$ is equivalent to $F(\Gamma)$.  

Let $B_1,\ldots,B_\mu$ be mutually disjoint 2-disks in $\mathbb{R}^3$ such that $\partial(B_1\cup\cdots\cup B_\mu)=L(O)$ and $\pi|_{B_1\cup\cdots\cup B_\mu}$ is an embedding. Let $B_1',\ldots,B_\lambda'$ be mutually disjoint 2-disks in $\mathbb{R}^3$ 
such that $\partial(B_1'\cup\cdots\cup B_\lambda')=L(O')$ and $\pi|_{B_1'\cup\cdots\cup B_\lambda'}$ is an embedding. Finally we define $F$ to be the surface constructed as follows:
$$F=(B_1'\cup\cdots\cup B_\lambda')\times\{-2\} \cup  (F'\cap(\mathbb{R}^3\times(-2,2)) ) \cup (B_1\cup\cdots\cup B_\mu)\times\{2\}. $$
It is in general position with respect to the projection $q : \mathbb R^4 \to \mathbb R^3, (x,y,z,t)\\ \mapsto (x,y,t)$. The broken surface diagram of $F$ obtained from  $q(F)$   is called a {\it broken surface diagram associated to $\Gamma$}, and denoted by  $\mathcal{B}(\Gamma)$.


\section{Quandle cocycle invariants of oriented surface-links}\label{sect-qcoc}

We recall quandle cocycle invariants of oriented surface-links from \cite{CJKLS}.

A {\it quandle} is a set $X$ with a binary operation $\ast:X\times X\rightarrow X$ satisfying that 
(i) for any $x\in X$, $x\ast x=x$, (ii) for any $x,y \in X$, there is a unique $u\in X$ such that $x=u\ast y$, and (iii)  for any $x,y,z\in X$, $(x\ast y)\ast z=(x\ast z)\ast (y\ast z)$.   
In (ii), the unique element $u$ is denoted by $x\ast \bar y$, and then $x=u\ast y=(x\ast\bar y)\ast y$. 
 
\begin{example}\label{ex-qua}

(1) The {\it dihedral quandle} of order $n$ is the set 
${R_n}=\{0,1,\ldots, n-1\}$ with the binary operation  $i\ast j=2j-i$ (mod $n)$ for each $i,j\in R_n$.   

(2) Let $S_4=\{0,1,2,3\}.$ Define a binary operation $\ast:S_4\times S_4\rightarrow S_4$ by
\begin{center}
\begin{tabular}{c|cccc}
*&0&1&2&3\\
\hline
0&0&2&3&1\\
1&3&1&0&2\\
2&1&3&2&0\\
3&2&0&1&3\\
\end{tabular}
\end{center} 
Then $S_4$ is a quandle, which is called the {\it tetrahedral quandle}.
 

(3) Let $G$ be a group. The {\it conjugation quandle,}  denoted by conj$(G)$, is $G$ with the operation $x\ast y=y^{-1}xy.$ 
  \end{example}

Let $X$ be a quandle. For each positive integer $n$, let $C_n^R(X)$ be the free abelian group generated by $n$-tuples $(x_1,\ldots,x_n)$ of elements of $X$.  We assume $C_n^R(X) = \{ 0 \}$ for $n \leq 0$. 
Define a homomorphism $\partial_n:C_n^R(X)\rightarrow C_{n-1}^R(X)$ by \begin{align*}
\partial_n(x_1,x_2,\ldots,x_n)&=\sum^{n}_{i=2} (-1)^i[(x_1,x_2,\ldots,x_{i-1},x_{i+1},\ldots,x_n)\\&-(x_1\ast x_i,x_2\ast x_i,\ldots,x_{i-1}\ast x_i,x_{i+1},\ldots,x_n)]\end{align*} for $n\geq2$ and $\partial_n=0$ for $n\leq1.$ Then $C^R_*(X)=\{C_n^R(X),\partial_n\}$ is a chain complex.
   Let $C_n^D(X)$ be the subset of $C_n^R(X)$ generated by $n$-tuples $(x_1,\ldots,x_n)$ with $x_i=x_{i+1}$ for some $i\in{1,\ldots,n-1}$ if $n\geq2;$ otherwise let $C_n^D(X)=0.$ Then $C_*^D(X)=\{C_n^D(X),\partial_n\}$ is a sub-complex of $C_*^R(X).$ Consider the quotient chain complex $C_*^Q(X)=\{C_n^Q(X),\partial_n \}$, where $C_n^Q(X)=C_n^R(X)/C_n^D(X)$.   
For an abelian group $A$,  we define chain and cochain complexes by 
$C_*^Q(X;A)=C^Q_*(X)\otimes A$ and 
$C_Q^*(X;A)=Hom(C_*^Q(X),A)$.  The homology and cohomology groups are denoted by 
$H_n^Q(X;A)$ and $H_Q^n(X;A)$, respectively.  
The cycle and boundary groups (or cocycle and coboundary groups, resp.) are denoted by $Z_n^Q(X;A)$ and $B_n^Q(X;A)$ (or $Z^n_Q(X;A)$ and $B^n_Q(X;A)$, resp.).  
We will omit the coefficient group $A$ as usual if $A=\mathbb{Z}.$

A homomorphism 
$\theta:C_3^R(X)\rightarrow A$ is regarded as a 3-cocycle of the cochain complex $C_Q^*(X;A)$, called a {\it quandle $3$-cocycle},  
if and only if  $\theta$ satisfies the following two conditions (where $A$ is written multiplicative):
\begin{itemize}
\item[$\bullet$] $\theta(x,x,y)=1$ and $\theta(x,y,y)=1$ for all $x,y\in X,$ where 1 is the identity element in $A$.
\item[$\bullet$] $\theta(x,z,w)\theta(x,y,z)\theta(x\ast z,y\ast z,w)=\theta(x,y,w)\theta(x\ast y,z,w)\theta(x\ast w, y\ast w, z\ast w)$ for each $x,y,z,w\in X.$
\end{itemize}

Let $\mathcal B$ be a broken surface diagram of an oriented surface-link $\mathcal L$, and let $S(\mathcal B)$ be the set of sheets of $\mathcal B$.  
Let $X$ be a quandle.  A {\it coloring} of $\mathcal B$ by $X$ is a map $\mathcal{C}: S(\mathcal B)\rightarrow X$ satisfying the condition that at each double point curve,  if the co-orientation of the over-sheet $y$ is from the under-sheet $x$ to $z$, then $\mathcal{C}(z)=\mathcal{C}(x)\ast\mathcal{C}(y).$ See the left of Figure \ref{fig-colb}.  Let ${\rm Col}_X(\mathcal B)$ denote the set of all colorings of $\mathcal B$ by $X$. 

    \begin{figure}[ht]
\begin{center}
\resizebox{0.55\textwidth}{!}{%
  \includegraphics{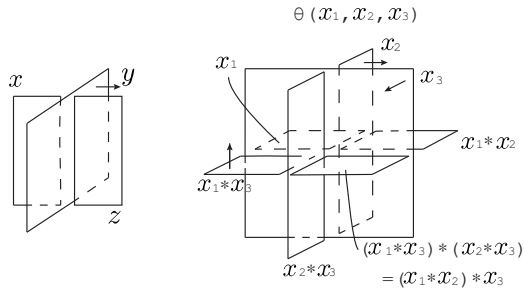} }
\caption{A double point curve and a triple point}\label{fig-colb}
\end{center}
\end{figure}

Let $\tau$ be a triple point of $\mathcal B$.  
The {\it sign} of $\tau$ is positive if the co-orientations of the top, the middle and the bottom sheets at $\tau$ in this order match the given (right-handed) orientation of $\mathbb{R}^3$. Otherwise, the sign is negative.  
There are eight complementary regions of $\mathcal B$ around $\tau$. (Some of them may be the same.)  There is a unique region such that the co-orientations of the sheets facing the region point from the region to the opposite regions.  We call this region the {\it source region} of $\tau$.  

For a 3-cocycle $\theta\in Z^3_Q(X;A)$, the quandle cocycle invariant $\Phi_\theta(\mathcal L)$ of an oriented surface-link 
$\mathcal L$ associated to $\theta$ is defined as follows.   Let $\mathcal B$ be a broken surface diagram of $\mathcal L$. Let $\mathcal{C}: S(\mathcal B)\rightarrow X$ be a coloring of $\mathcal B$. Let $\tau$ be a triple point of $\mathcal B$ and let $x_1, x_2$, and $x_3$ be colors of the bottom, the middle, and the top sheets facing the source region of $\tau$, respectively. Let $\epsilon(\tau)$ denote the sign of $\tau$. 
See Figure~\ref{fig-colb}, where $\epsilon(\tau)=1$.  
The {\it (Boltzman) weight} $B_\theta(\tau,\mathcal{C})$ at $\tau$ with respect to $\mathcal{C}$ is defined to be $$B_\theta(\tau,\mathcal{C})=\theta(x_1, x_2, x_3)^{\epsilon(\tau)}.$$

The {\it partition function} or {\it state-sum} of $\mathcal B$ (associated to  $\theta$) is 
\begin{align*}\Phi_\theta(\mathcal B)=\sum_{\mathcal{C}\in {\rm Col}_X(\mathcal B)}\prod_{\tau\in T(\mathcal B)}B_\theta(\tau,\mathcal{C})\in \mathbb Z[A],\end{align*}
where  $T(\mathcal B)$ is the set of all triple points in $\mathcal B.$

\begin{theorem}[\cite{CJKLS}]\label{thm-cocycleb}
Let $\mathcal L$ be an oriented surface-link and let $\mathcal B$ be a broken surface diagram of $\mathcal L$. The partition function $\Phi_\theta(\mathcal B)$ does not depend on the choice of $\mathcal B$. Thus it is an invariant of $\mathcal L$. \end{theorem}

We call $\Phi_\theta(\mathcal B)$ the {\it quandle cocycle invariant} of $\mathcal L$ associated to $\theta$, and denote it by $\Phi_\theta(\mathcal L).$


\section{How to compute quandle cocycle invariants from marked graph diagrams}\label{sect-qcocm}

In this section we introduce a method of computing quandle cocycle invariants from  marked graph diagrams.

Let $\Gamma$ be an oriented marked graph diagram and let $V(\Gamma)$ denote the set of all marked vertices of $\Gamma$.   By an {\it arc} of $\Gamma$ we mean  a connected component of $\Gamma \setminus V(\Gamma)$.  (At a crossing of $\Gamma$ the under-arcs are assumed to be cut.)  Let $A(\Gamma)$ denote the set of arcs of $\Gamma$.   Since $\Gamma$ is oriented, we assume that it is co-oriented: 
The co-orientation of an arc of $\Gamma$ satisfies that the pair (orientation, co-orientation) matches the (right-handed) orientation of the plane. At a crossing, if the pair of the co-orientation of the over-arc and that of the under-arc matches the (right-handed) orientation of the plane, then the crossing is called {\it positive}; otherwise it is {\it negative}. The crossing in (a) of Figure~\ref{fig-col} is positive and that in (b) is negative.

\begin{definition}\label{defn-qc-1}
Let $X$ be a quandle and let $\Gamma$ be an oriented marked graph diagram. A {\it  coloring} of $\Gamma$ by $X$ is a map $\mathcal{C}: A(\Gamma) \rightarrow X$ satisfying the following conditions (1) and (2):
\begin{itemize}
\item [(1)] For each crossing $c$, let $s_2$ be the over-arc and let $s_1$ and $s_3$ be the under-arcs  as shown in (a) or (b) of Figure~\ref{fig-col} such that the co-orientation of $s_2$ points from $s_1$ to  $s_3$. Then 
$\mathcal C(s_3)= \mathcal C(s_1) \ast \mathcal C(s_2)$.  

(In this case, $s_1$ is called the {\it source arc} and $s_3$ is called the {\it target arc} at $c$. The quandle element $\mathcal{C}(s_i)$ is called a {\it color} of the arc $s_i$.)  

\item[(2)] For each marked vertex $v$, let $s_1, s_2, s_3$ and $s_4$ be the arcs of $\Gamma$ as shown in (c) or (d) of Figure~\ref{fig-col}.
Then $\mathcal C(s_1)=\mathcal C(s_2)=\mathcal C(s_3)=\mathcal C(s_4)$.  

\begin{figure}[ht]
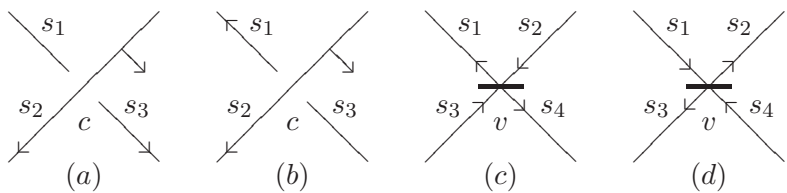

\centerline{
\xy (-10,10);(-2,2) **@{-}, 
(10,-10);(2,-2) **@{-}, 
(10,10);(-10,-10) **@{-},  
(5,5);(8,2) **@{-},
(-8,-8) *{\llcorner}, (7.4,2.7) *{\lrcorner},
(8,-8) *{\lrcorner}, (7,-3) *{s_3},(-4,8) *{s_1},(-7,-3) *{s_2},(0,-5) *{c},
(0,-12)*{(a)},
\endxy 
 \qquad
\xy (-10,10);(-2,2) **@{-}, 
(10,-10);(2,-2) **@{-}, 
(10,10);(-10,-10) **@{-},  
(5,5);(8,2) **@{-},
(-8,-8) *{\llcorner}, (7.4,2.7) *{\lrcorner},
(-8,7.5) *{\ulcorner}, 
(7,-3) *{s_3},(-4,8) *{s_1},(-7,-3) *{s_2},(0,-5) *{c},
(0,-12)*{(b)},
\endxy
\qquad
\xy (-10,10);(10,-10) **@{-}, 
(10,10);(-10,-10) **@{-}, 
(3,3.2)*{\llcorner}, 
(-3,-3.4)*{\urcorner}, 
(-2.5,2)*{\ulcorner},
(2.5,-2.4)*{\lrcorner}, 
(3,-0.2);(-3,-0.2) **@{-},
(3,0);(-3,0) **@{-}, 
(3,0.2);(-3,0.2) **@{-}, 
(7,-3) *{s_4},(-4,8) *{s_1},(-7,-3) *{s_3},(4,8) *{s_2},(0,-5) *{v},
(0,-12)*{(c)},
\endxy
 \qquad
\xy (-10,10);(10,-10) **@{-}, 
(10,10);(-10,-10) **@{-},  
(2.5,2)*{\urcorner}, 
(-2.5,-2.2)*{\llcorner}, 
(-3.2,3)*{\lrcorner},
(3,-3.4)*{\ulcorner},
(3,-0.2);(-3,-0.2) **@{-},
(3,0);(-3,0) **@{-}, 
(3,0.2);(-3,0.2) **@{-}, 
(7,-3) *{s_4},(-4,8) *{s_1},(-7,-3) *{s_3},(4,8) *{s_2},(0,-5) *{v},
(0,-12)*{(d)},
\endxy 
}
\vskip.1cm
\caption{Labels at a crossing}\label{fig-col}
\end{figure}

\end{itemize}

\end{definition}

We denote by ${\rm Col}_X(\Gamma)$ the set of colorings of $\Gamma$ by $X$.  

\begin{theorem}\label{thm-col}
Let $\mathcal L$ be an oriented surface-link. Let $\Gamma$ and $\mathcal B$ be a marked graph diagram and a broken surface diagram presenting $\mathcal L$, respectively. Then there is a bijection from ${\rm Col}_X(\Gamma)$ to ${\rm Col}_X(\mathcal B)$. 
\end{theorem}

\begin{proof}
The fundamental quandle $Q(\Gamma)$ is defined by a quandle generated by $A(\Gamma)$ and the defining relations $s_3 = s_1 \ast s_2$ for $s_1, s_2, s_3$ as in (a) or (b) in Figure~\ref{fig-col} and $s_1 = s_2 = s_3 = s_4$ for $s_1, \dots, s_4$ as in (c) or (d).  Without loss of generality, we may assume that $\mathcal B$ is a broken surface diagram associated to $\Gamma$.  Then by 
the same argument with \cite{As} we see that there is a natural isomorphism from  
 $Q(\Gamma)$ to the fundamental quandle $Q(\mathcal B)$ of $\mathcal B$.   
Since ${\rm Col}_X(\Gamma)$ is identified with ${\rm Hom}(Q(\Gamma),X)$ and 
${\rm Col}_X(\mathcal B)$ is identified with 
${\rm Hom}(Q(\mathcal B),X)$, we have a bijection from ${\rm Col}_X(\Gamma)$ to ${\rm Col}_X(\mathcal B)$. 
\end{proof}

Let $\Gamma$ be a marked graph diagram of an oriented surface-link $\mathcal L$ and $\Gamma_+$ the positive  resolution of $\Gamma.$ Let $\Gamma_+=D_1 \rightarrow D_2 \rightarrow\cdots\rightarrow D_r=O$  be a sequence of link diagrams from $\Gamma_+$ to a trivial link diagram $O$ related by ambient isotopies of $\mathbb R^2$ and oriented Reidemeister moves.   Let $I^3_+=\{i \mid D_{i}\rightarrow D_{i+1} \text{ is a move of type }  R_3\}$.  For each $i \in I^3_+$, let $B_{(i)}$ be a disk in $\mathbb R^2$ where the move $D_i\rightarrow D_{i+1}$  is applied.  

Similarly, let $\Gamma_-$ be the negative  resolution of $\Gamma$ and $\Gamma_-=D_1' \rightarrow D_2' \rightarrow\cdots\rightarrow D_s'=O'$ a sequence  of link diagrams from $\Gamma_-$ to a trivial link diagram $O'$ related by ambient isotopies of $\mathbb R^2$ and Reidemeister moves. Let $I^3_-=\{j \mid D_{j}'\rightarrow D_{j+1}' \text{ is a move of type } R_3\}$.  For each $j \in I^3_-$, let $B_{(j)}'$ be a disk in $\mathbb R^2$ where  the move $D_{j}'\rightarrow D_{j+1}'$ is applied.

We define two functions $\epsilon_{tm}$ and $\epsilon_{b}$ from the disjoint union $I^3_+ \amalg {I^3_-}$ to $\{\pm1\}$ as follows: 

Let $i\in I^3_+$ (or $i\in I^3_-$, resp.) and 
let $c$ 
be the crossing between the top arc and the two middle arcs in  $D_{i}\cap B_{(i)}$ (or $D_{i}'\cap B_{(i)}'$, resp.) 
and let $n_1$  
be the co-orientation of the bottom arc. Define $\epsilon_{tm}(i)$ and $\epsilon_{b}(i)$ for $i\in I^3_+\amalg I^3_-$ by
\begin{align}
&\epsilon_{tm}(i)={\rm sign}(c),\label{eq-etmi}\\
&\epsilon_{b}(i)=\begin{cases}
1 \hskip 0.9cm \text{if $n_1$ points from }c,\\
-1 \hskip 0.6cm \text{otherwise}. \\
\end{cases}\label{eq-ebi}
\end{align}
%
%

\begin{definition}\label{def-weightm}
Let $\Gamma$ be a marked graph diagram of an oriented surface-link $\mathcal L.$ Let $\mathcal{C}:A(\Gamma)\rightarrow X$ be a coloring of  $\Gamma$ and let $\theta\in Z^3_Q(X;A).$

Let $i\in I^3_+ \amalg {I^3_-}$. Let $R$ be the source region of the crossing $c$, i.e., the quadrant from which all co-orientations of the top arc and the middle arc point outwards. Let $R'$ be the opposite region of $R$ with respect to the top arc.  The {\it(Boltzman) weight} $B_\theta(i,\mathcal{C})$ at $i$ with respect to $\mathcal C$ is defined by $$B_\theta(i,\mathcal{C})=\theta(x_1,x_2,x_3)^{\epsilon_{tm}(i)\epsilon_{b}(i)},$$ where $x_2$ and $x_3$ are the colors of the middle arc and the top semi-arc facing $R$, respectively,  and $x_1$ is the color of the bottom semi-arc which is in $R$ or $x_1$ is the element with $x_1=a\ast\overline{x_3},$ where $a$ is the color of the bottom semi-arc which is in $R'$. See Figure \ref{fig-lab1}.

 \begin{figure}[ht]
\begin{center}
\resizebox{0.6\textwidth}{!}{%
  \includegraphics{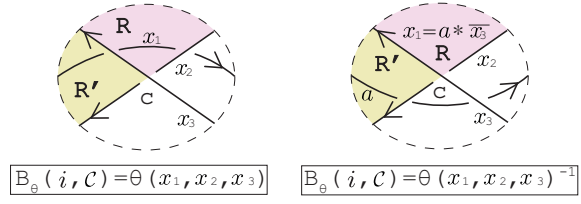} }
\caption{A (Boltzman) weight  at $i\in I^3_+ \amalg {I^3_-}$}\label{fig-lab1}
\end{center}
\end{figure}

 \end{definition}

\begin{definition}\label{def-cocyclem}
Let $\Gamma$ be a marked graph diagram of an oriented surface-link $\mathcal L.$ 
The {\it partition function} or {\it state-sum} of $\Gamma$ (associated to  $\theta$) is 
\begin{align*}
\Phi_\theta(\Gamma)=\sum_{\mathcal{C}\in {\rm Col}_X(\Gamma)}\biggl(\prod_{i\in I^3_+ }B_\theta(i,\mathcal{C})\prod_{j\in I^3_- }B_\theta(j,\mathcal{C})^{-1}\biggr).\end{align*}
\end{definition}

\begin{theorem}\label{thm-cocyclem}
Let $\mathcal L$ be an oriented surface-link and $\Gamma$ a marked graph diagram of $\mathcal L$. Then for any $\theta\in Z^3_Q(X;A),$ $\Phi_\theta(\mathcal L)=\Phi_\theta(\Gamma)$. 
\end{theorem}

\begin{proof}
Let $\mathcal B=\mathcal{B}(\Gamma)$ be a broken surface diagram associated to $\Gamma$. It is sufficient to prove that $\Phi_\theta(\Gamma)=\Phi_\theta(\mathcal B)$. 

Since there is a natural bijection between  
${\rm Col}_X(\Gamma)$ and ${\rm Col}_X(\mathcal B)$ (as in the proof of Theorem~\ref{thm-col}), it 
suffices to show  the following claim.  

\smallskip
 {\bf Claim}: For each coloring $\mathcal C\in {\rm Col_X(\Gamma),}$
 $$\prod_{i\in I^3_+ }B_\theta(i,\mathcal{C})\prod_{j\in I^3_- }B_\theta(j,\mathcal{C})^{-1}=\prod_{\tau\in T(\mathcal B)}B_\theta(\tau,\mathcal C),$$
 where $\mathcal C\in {\rm Col_X(\mathcal B)}$ in the right hand side  is the corresponding  coloring.  
 
 \bigskip
{\it{\bf Proof of Claim.}}    Let $\mathcal B^i_j=\mathcal B\cap (\mathbb R^2\times [t_j',t_i])$ for $i=1, \ldots, r$ and $j=1, \ldots, s.$  Let $\phi:(\mathbb R^2,\Gamma_0)\rightarrow (\mathbb R^2\times [t_1',t_1], \mathcal B^1_1)$ be the natural embedding at $t=0$ as in Figure \ref{fig-ne}. The vertices of $\Gamma_0$ correspond to the saddle points in $\mathcal B^1_1$ and the crossings of  $\Gamma_0$ correspond to the intersection of $\mathbb R^2 \times \{0\}$ and the double point curves in $\mathcal B^1_1$. There are no triple points in $\mathcal B^1_1$. 
   
 \begin{figure}[ht]
\begin{center}
\resizebox{0.7\textwidth}{!}{%
  \includegraphics{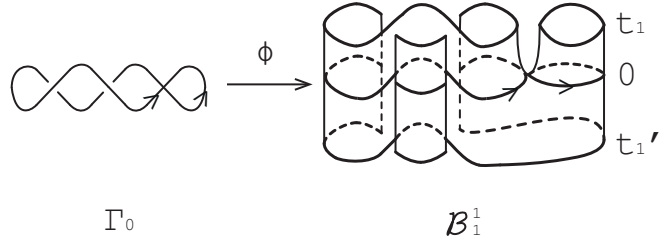} }
\caption{$\phi:(\mathbb R^2,\Gamma_0)\rightarrow (\mathbb R^2\times [t_1',t_1], \mathcal B^1_1)$}\label{fig-ne}
\end{center}
\end{figure}

Let $\mathcal B_i=\mathcal B\cap(\mathbb R^2\times [t_{i},t_{i+1}])$ for $i=1,\ldots, r-1$ and 
$\mathcal B_{j}'=\mathcal B\cap(\mathbb R^2\times [t_{j+1}',t_j'])$ for  $j=1,\ldots, s-1.$ Note that 
$T(\mathcal B)=\bigl(\overset{r-1}{\underset{i=1}{\cup}}T(\mathcal B_i)\bigr)\cup\bigl(\overset{s-1}{\underset{j=1}{\cup}}T(\mathcal B_j')\bigr)$.  

If the move $D_i\rightarrow D_{i+1}$ is an ambient isotopy of $\mathbb R^2$, then $D_{i}\times [t_{i},t_{i+1}]\cong \mathcal B_i$, and 
there are no triple points in $\mathcal B_i.$

Suppose that the move $D_i\rightarrow D_{i+1}$ is a  Reidemeister move. Since $D_{i}\setminus B_{(i)}$ and $D_{i+1}\setminus B_{(i)}$ are identical, there are no triple points in $\mathcal B_i\setminus M_{(i)}$ and we have $T(\mathcal B_i)=T(M_{(i)}),$ where $M_{(i)}$ is a subset of $B_{(i)}\times I$ determined by $\pi(M_{(i)}\cap (B_{(i)}\times \{t\}))=\pi(f_t^{(i)}(L(D_i)))\cap B_{(i)}$ for $t\in I$ and a homeomorphism $f_t^{(i)}:\mathbb R^3\rightarrow \mathbb R^3$ satisfying $f_0^{(i)}={\rm id}$ and $f_1^{(i)}(L(D_i))=L(D_{i+1}).$

If the move $D_i\rightarrow D_{i+1}$ is of  type $R_1$ or $R_2$, then there are no triple points in 
 $M_{(i)}$.  See Figures \ref{fig-m1} and \ref{fig-m2}.  

If the move $D_{i}\rightarrow D_{i+1}$ is of type $R_3$, then there is a triple point $\tau_i$ in $M_{(i)}$ as  in Figure \ref{fig-trip} and $T(\mathcal B_i)=\{\tau_i\}.$ Then $\overset{r-1}{\underset{i=1}{\cup}}T(\mathcal B_i) = \{\tau_i \mid i \in I^3_+\}.$ 

 \begin{figure}[ht]
\begin{center}
\resizebox{1\textwidth}{!}{%
  \includegraphics{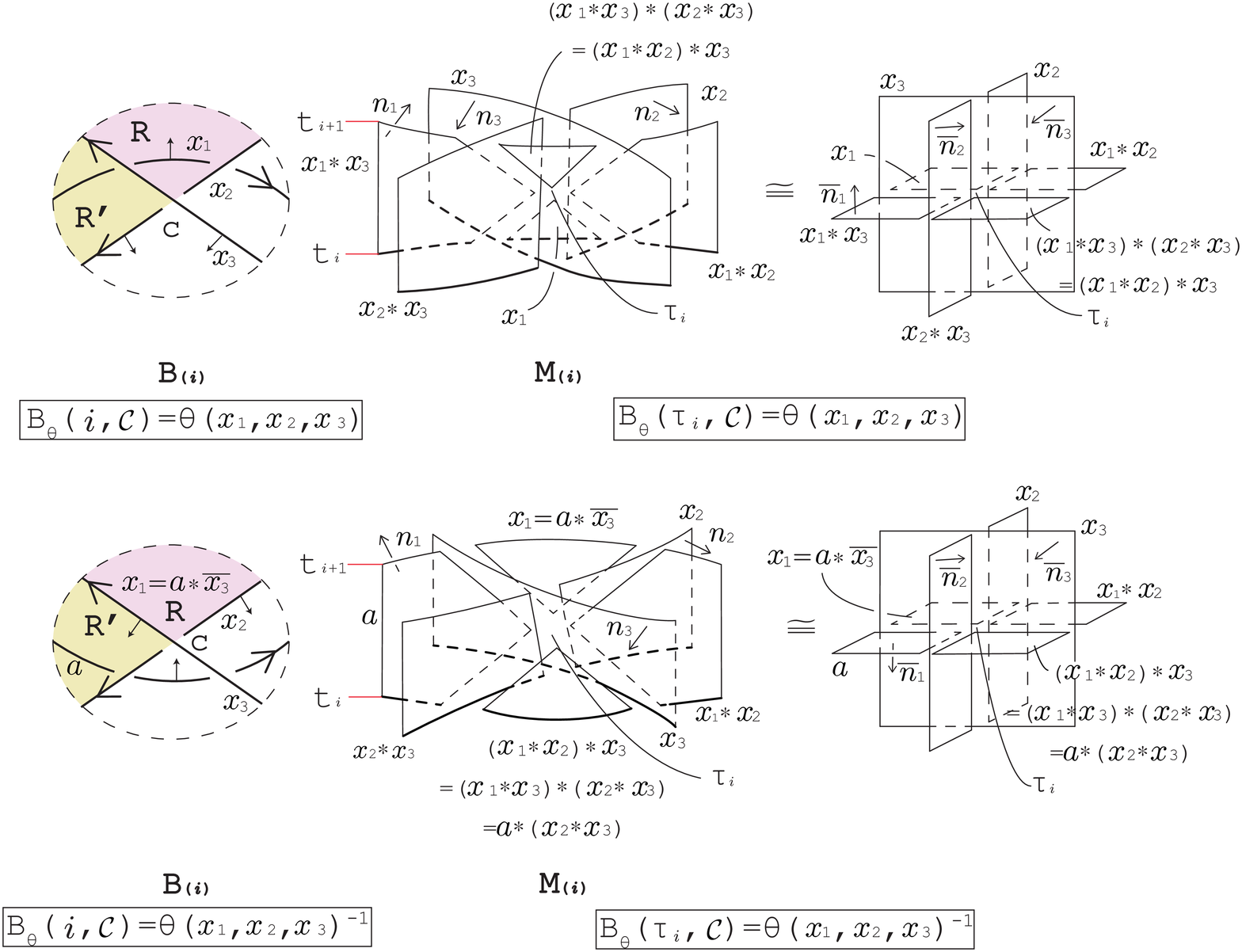} }
\caption{A (Boltzman) weight}\label{fig-trip}
\end{center}
\end{figure}

  Similarly, suppose that the move $D_j'\rightarrow D_{j+1}'$ is a  Reidemeister move and $M_{(j)}'$ is a subset of $B_{(j)}'\times I$ determined by $\pi(M_{(j)}'\cap (B_{(j)}'\times \{t\}))=\pi(g_t^{(j)}(L(D_j')))\cap B_{(j)}',$ where $g_t^{(j)}:\mathbb R^3\rightarrow \mathbb R^3$ is a homeomorphism  satisfying $g_0^{(j)}={\rm id}$ and $g_1^{(j)}(L(D_j'))=L(D_{j+1}')$ for $t\in I.$  There is a triple point $\tau_j'\in M_{(j)}'$ for  $j \in I^3_-$.   We have that  $\overset{s-1}{\underset{j=1}{\cup}}T(\mathcal B_j')=\{\tau_j' \mid j \in I^3_-\}.$  
  
Now we have 
   \begin{align}\label{eq-tr}
   T(\mathcal B)= \{ \tau_i \mid i \in I^3_+\} \cup \{ \tau_j' \mid j \in I^3_-\}. \end{align}

Let $i\in I^3_+$, i.e.,  $D_i\rightarrow D_{i+1}$ is a move of type $R_3$ and let $\tau_i$ be the  triple point in $M_{(i)}.$ Let $n_1$, $n_2$  and $n_3$ be the co-orientations of the bottom, the middle and the top arcs of $D_i$  in $B_{(i)}$, respectively. By an ambient isotopy, we deform $M_{(i)}$ in $B_{(i)} \times I$ to the standard  form of the neighborhood of the triple point $\tau_i$ as in Figure~\ref{fig-trip}.  
  Let $\bar{n}_1,$ $\bar{n}_2,$ and $\bar{n}_3$ be the normal vectors corresponding to $n_1,$ $n_2,$ and $n_3$, respectively. Without loss of generality, we may assume $\bar{n}_3= {\bf e}_1$, $\bar{n}_2= \epsilon \/ {\bf e}_2$ and $\bar{n}_1= \epsilon' \/ {\bf e}_3$ for some $\epsilon,\epsilon'\in\{1,-1\}$.  Here ${\bf e}_1 = (1,0,0)$, ${\bf e}_2 = (0,1,0)$ and ${\bf e}_3 = (0,0,1)$.  
See Figure~\ref{fig-trip}.

Let $c$ be the crossing between the top and the middle arcs in $B_{(i)}$. It is clear from Figure \ref{fig-trip} that $\epsilon={\rm sign}(c)$. By (\ref{eq-etmi}), $\epsilon={\rm sign}(c)=\epsilon_{tm}(i).$ Hence $\bar{n}_2= \epsilon_{tm}(i) \/ {\bf e}_2$.

The sign $\epsilon'$ depends on the co-orientation $n_1$ of the bottom arc. If $n_1$ points from $c$, then $\epsilon'=1$. If $n_1$ points toward $c,$ then $\epsilon'=-1.$ So, by (\ref{eq-ebi}), $\epsilon'=\epsilon_b(i)$ and hence $\bar{n}_1= \epsilon_b(i) \/ {\bf e}_3$. 

On the other hand, by definition, the sign $\epsilon(\tau_i)$ of the triple point $\tau_i$ is positive if the co-orientations of the top, the middle and the bottom sheets in this order match the given (right-handed) orientation of $\mathbb{R}^3$. Otherwise, the sign $\epsilon(\tau_i)$ is negative. This gives
\begin{align*}
\epsilon(\tau_i)=\begin{cases}
1 \hskip 1.3cm \text{if } (\bar{n}_3,\bar{n}_2,\bar{n}_1)\in A,\\
-1 \hskip 1cm \text{if }(\bar{n}_3,\bar{n}_2,\bar{n}_1)\in B,\\
\end{cases}
\end{align*}
where 
$A=\{( {\bf e}_1, {\bf e}_2, {\bf e}_3 ), ( {\bf e}_1, -{\bf e}_2, -{\bf e}_3) \}$ and 
$B=\{( {\bf e}_1, -{\bf e}_2, {\bf e}_3 ), ( {\bf e}_1, {\bf e}_2, -{\bf e}_3 ) \}.$  Therefore for each $i\in I^3_+,$ 
\begin{align}\label{eq-taui}\epsilon(\tau_i)=\epsilon_{tm}(i)\epsilon_b(i).\end{align}

Let $j\in I^3_-$, i.e.,   $D_j'\rightarrow D_{j+1}'$ is a move of type $R_3$. Let $\tau_j'$ be the triple point in $M_{(j)}'$.  
 Let $n_1$, $n_2$  and $n_3$ be the co-orientations of the bottom, the middle and the top arcs of $D_j'$ in $B_{(j)}'$, respectively. By an ambient isotopy, we deform $M_{(j)}'$ to the standard  form of the neighborhood of the triple point $\tau_j'.$ Let $\bar{n}_1,$ $\bar{n}_2,$ and $\bar{n}_3$ be the co-orientations corresponding to $n_1,$ $n_2,$ and $n_3$, respectively. Without loss of generality, we may assume $\bar{n}_3= {\bf e}_1$,  $\bar{n}_2=  \epsilon \/ {\bf e}_2$ and $\bar{n}_1= \epsilon' \/ {\bf e}_3$ for some $\epsilon,\epsilon' \in\{1,-1\}$.

Let $c$ be the crossing between the top and the middle arcs in $B_{(j)}'$. It is easily seen that $\epsilon={\rm sign}(c)$ (cf. Figure \ref{fig-trip}). By (\ref{eq-etmi}), $\epsilon={\rm sign}(c)=\epsilon_{tm}(j).$ Hence $\bar{n}_2= \epsilon_{tm}(j) \/ {\bf e}_2$. 

The sign $\epsilon'$ depends on the co-orientation $n_1$ of the bottom arc. If $n_1$ points from $c$, then $\epsilon'=-1$. If $n_1$ points toward $c,$ then $\epsilon'=1.$ So, by (\ref{eq-ebi}), $\epsilon'=-\epsilon_b(j)$ and hence $\bar{n}_1= -\epsilon_b(j)\/ {\bf e}_3$.

On the other hand, by definition, $\epsilon(\tau_j')=1$ if $\bar{n}_3, \bar{n}_2,$ and $\bar{n}_1$ in this order match the given (right-handed) orientation of $\mathbb{R}^3$. Otherwise, $\epsilon(\tau_j')=-1$. This gives 
\begin{align*}
\epsilon(\tau_j')=\begin{cases}
1 \hskip 1.3cm \text{if } (\bar{n}_3,\bar{n}_2,\bar{n}_1)\in B,\\
-1 \hskip 1cm \text{if }(\bar{n}_3,\bar{n}_2,\bar{n}_1)\in A.\\
\end{cases}
\end{align*} Therefore for each $j\in I^3_-,$ 
\begin{align}\label{eq-tauj}\epsilon(\tau_j')=-\epsilon_{tm}(j)\epsilon_b(j).\end{align}

We show that for each $i \in I^3_+$, $B_\theta(i,\mathcal C)=B_\theta(\tau_i,\mathcal C)$  and that for each $j \in I^3_-$, $B_\theta(j,\mathcal C)=B_\theta(\tau_j',\mathcal C)^{-1}$.  

Let $i\in I^3_+$ (or $j \in I^3_-$).  
There are two cases: The bottom arc meets the source region of the crossing $c$ or not 
(see Figure~\ref{fig-trip}). In this proof, we denote $M_{(i)}$ (or $M_{(j)}'$) by $M$ and $[t_i,t_{i+1}]$ (or $[t_{j+1}',t_j']$) by $I$.  

{\bf Case I}: Consider $i\in I^3_+$ (or $j \in  I^3_-$) 
such that the bottom arc in $B_{(i)}$ (or $B_{(j)}'$) hits the source region $R$ of $c$.

The top (or the middle, resp.) sheet in $M$ corresponds to the top (or the middle, resp.) arc times $I$. As shown in Figure \ref{fig-trip}, $R\times I$ is divided into two ($3$-dimensional) regions 
by the bottom sheet whose color is $x_1$.  One of them is the source region $\mathcal R$ of the triple point $\tau_i$ (or $\tau_j'$).    The colors of the top arc and the middle arc facing the source region $R$ of $c$ are the colors of the top and the middle sheets facing $\mathcal R$.  From the equalities (\ref{eq-taui}) and (\ref{eq-tauj}), we see that 
$B_\theta(i,\mathcal C)=\theta(x_1,x_2,x_3)^{\epsilon_{tm}(i)\epsilon_b(i)}=\theta(x_1,x_2,x_3)^{\epsilon(\tau_i)}=B_\theta(\tau_i,\mathcal C)$ and 
$B_\theta(j,\mathcal C)=\theta(x_1,x_2,x_3)^{\epsilon_{tm}(j)\epsilon_b(j)}=\theta(x_1,x_2,x_3)^{-\epsilon(\tau_j')}=B_\theta(\tau_j',\mathcal C)^{-1}$.  

{\bf Case II}: Consider $i\in I^3_+$ ($j \in  I^3_-$) such that the bottom arc in $B_{(i)}$ (or $B_{(j)}'$) does not meet the source region $R$ of $c$.

Similar to the case I, the second and third coordinates of $B_\theta(i,\mathcal C)$ 
(or $B_\theta(j,\mathcal C)$) 
are the same as those of $B_\theta(\tau_i,\mathcal C)$ (or 
$B_\theta(\tau_j',\mathcal C)$).  
As illustrated in Figure \ref{fig-trip}, $R'\times I$ is divided into two ($3$-dimensional) regions 
by the bottom sheet whose color is $a$, where $a$ is the color of the bottom arc in $R'$.   Since the co-orientation of the top sheet is from $R\times I$ to $R'\times I$,  $a=x_1\ast x_3.$ Thus $x_1=a\ast\overline{x_3}.$ Therefore 
$B_\theta(i,\mathcal C)=\theta(x_1,x_2,x_3)^{\epsilon_{tm}(i)\epsilon_b(i)}=\theta(x_1,x_2,x_3)^{\epsilon(\tau_i)}=B_\theta(\tau_i,\mathcal C)$ 
and 
$B_\theta(j,\mathcal C)=B_\theta(\tau_j',\mathcal C)^{-1}$.  
This completes the proof of Claim and hence the proof of Theorem~\ref{thm-cocyclem}.

\end{proof}

%
%
%
%

\begin{example}\label{ex-coc102}

We consider the oriented marked graph diagram $10_2$ of the  2-twist spun trefoil $\mathcal L$ in Figure \ref{fig-ytab}.
Let $$\theta={\chi_{0,1,0}}^{-1}\chi_{0,2,0}{\chi_{0,2,1}}^{-1}\chi_{1,0,1}\chi_{1,0,2}\chi_{2,0,2}\chi_{2,1,2}\in Z^3_Q(R_3;\mathbb Z_3),$$ where $\chi_{x,y,z}(a,b,c)=u$ if $(x,y,z)=(a,b,c),$ $\chi_{x,y,z}(a,b,c)=1$ otherwise, and $\mathbb Z_3= \langle u \mid u^3=1 \rangle$ is the cyclic group of order $3$.  Consider sequences of link  diagrams from the positive and negative resolutions to trivial link diagrams are as shown in Figures~\ref{fig-102+} and \ref{fig-102-}, respectively. Then $I^3_+=\phi$ and ${I^3_-}=\{2,3,4,5,8,10\}.$ The (Boltzman) weights are 
$B_\theta(2,\mathcal{C})=\theta(y,y,x)=1,$  
$B_\theta(3,\mathcal{C})=\theta(x*y,x,x)^{-1}=1,$  
$B_\theta(4,\mathcal{C})=\theta(x,x*y,x)^{-1},$ 
$B_\theta(5,\mathcal{C})=\theta(x,y,x)^{-1},$  
$B_\theta(8,\mathcal{C})=\theta(y,x,y)$ and 
$B_\theta(10,\mathcal{C})=\theta(x*y,x,x*y)$ for $x,y\in R_3.$ 
Therefore 
\begin{align*}\Phi_\theta(\mathcal L)&=\sum_{(x,y) \in R_3\times R_3} \theta(x,x*y,x)\theta(x,y,x)\theta(y,x,y)^{-1}\theta(x*y,x,x*y)^{-1}\\
&=3+6u.
\end{align*}
This matches the computation in \cite{CJKLS}.

 \begin{figure}
\begin{center}
\resizebox{1\textwidth}{!}{%
  \includegraphics{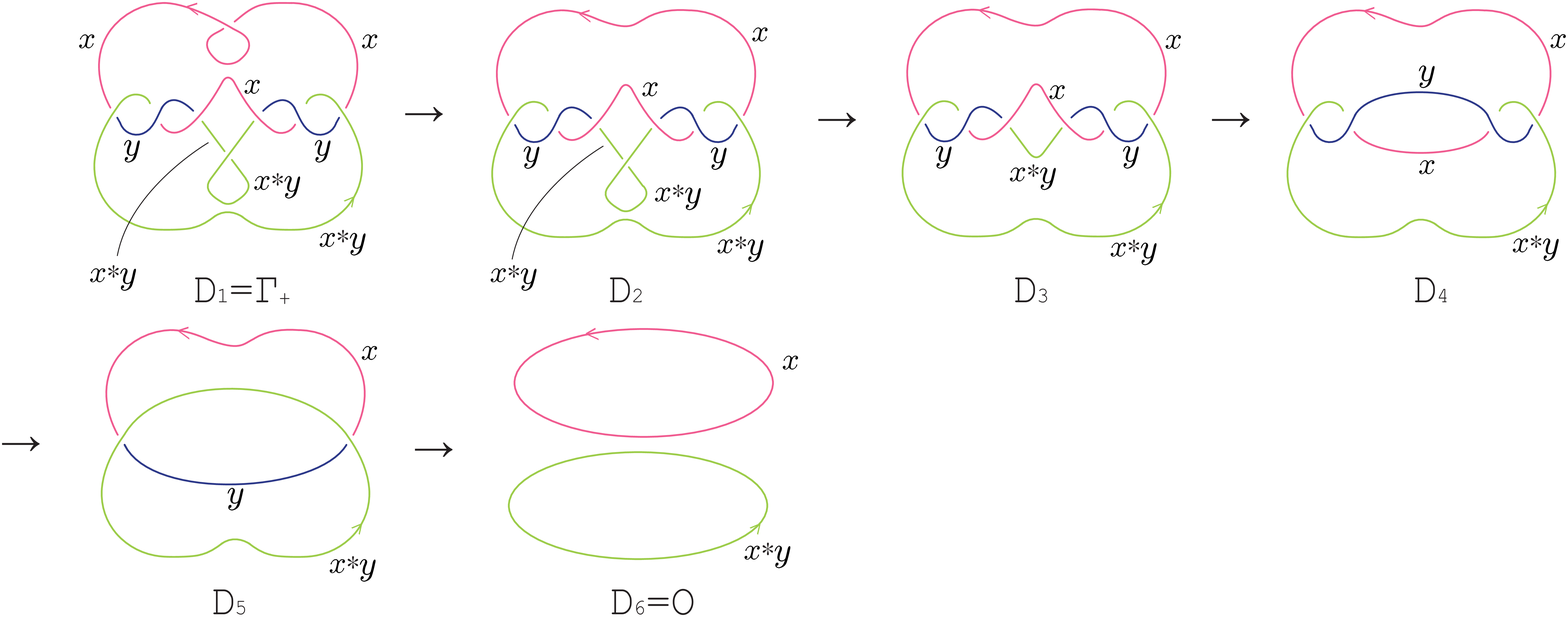} }
\caption{A sequence of link diagrams for positive resolution of $10_2$}\label{fig-102+}
\end{center}
\begin{center}
\resizebox{1\textwidth}{!}{%
  \includegraphics{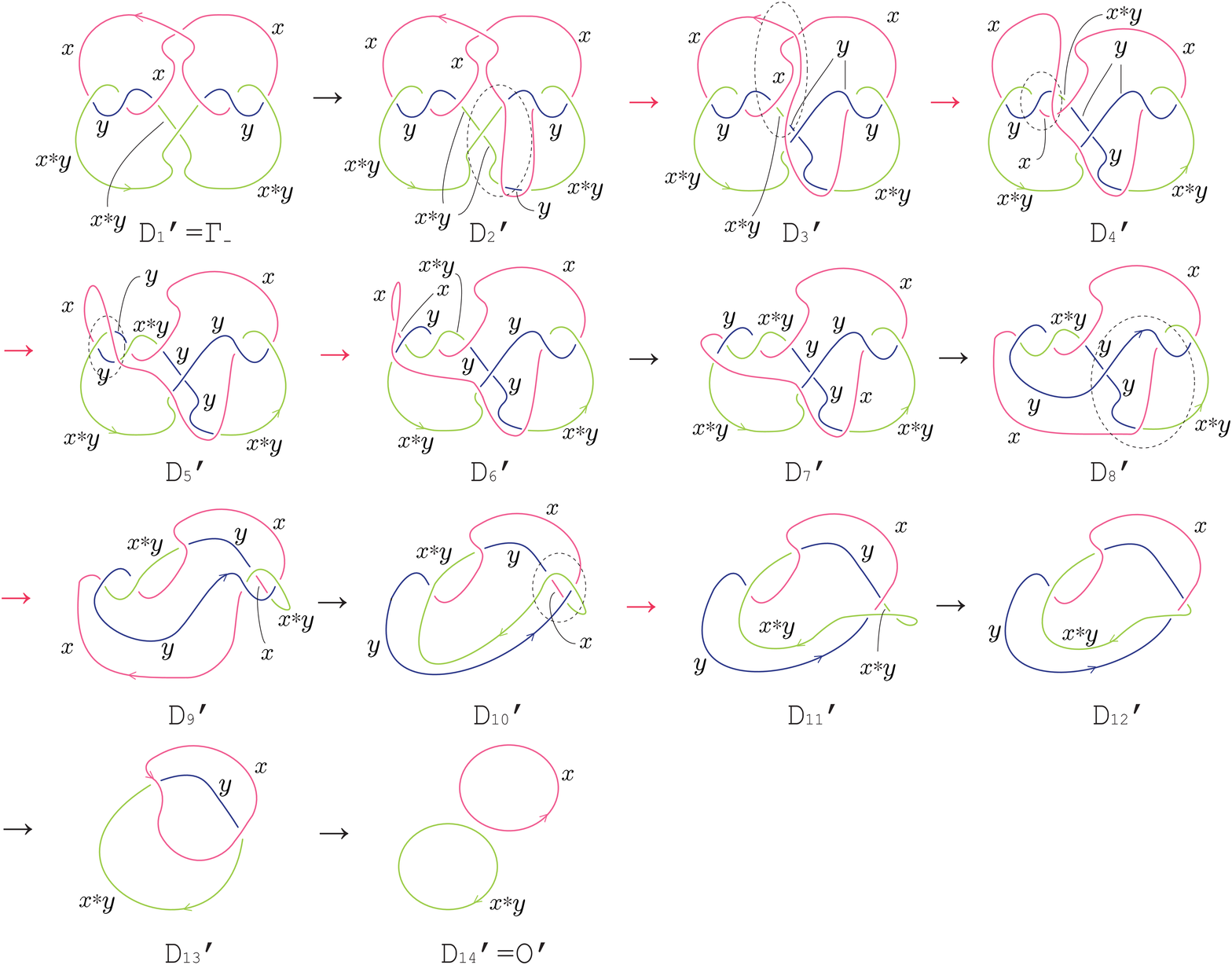} }
\caption{A sequence of link diagrams for negative resolution of $10_2$}\label{fig-102-}
\end{center}
\end{figure}

\end{example}

For 
a surface-link $\mathcal L$, the {\it ch-index} $\chi(\mathcal L)$  is defined by $\min_{\Gamma}\chi(\Gamma),$ where $\Gamma$ is a marked graph diagram presenting $\mathcal L$ and $\chi(\Gamma)$  is the sum of the number of crossings of $\Gamma$ and that of vertices of $\Gamma$.

\begin{example}\label{ex-cocy}
 Let $\mathcal L$ be an oriented surface-link with $\chi(\mathcal L)\leq 10$ presented by a marked graph diagram in 
Figure~\ref{fig-ytab} (see \cite{Yo}, for more details). 
 
  \begin{figure}[ht]
\begin{center}
\resizebox{0.8\textwidth}{!}{%
  \includegraphics{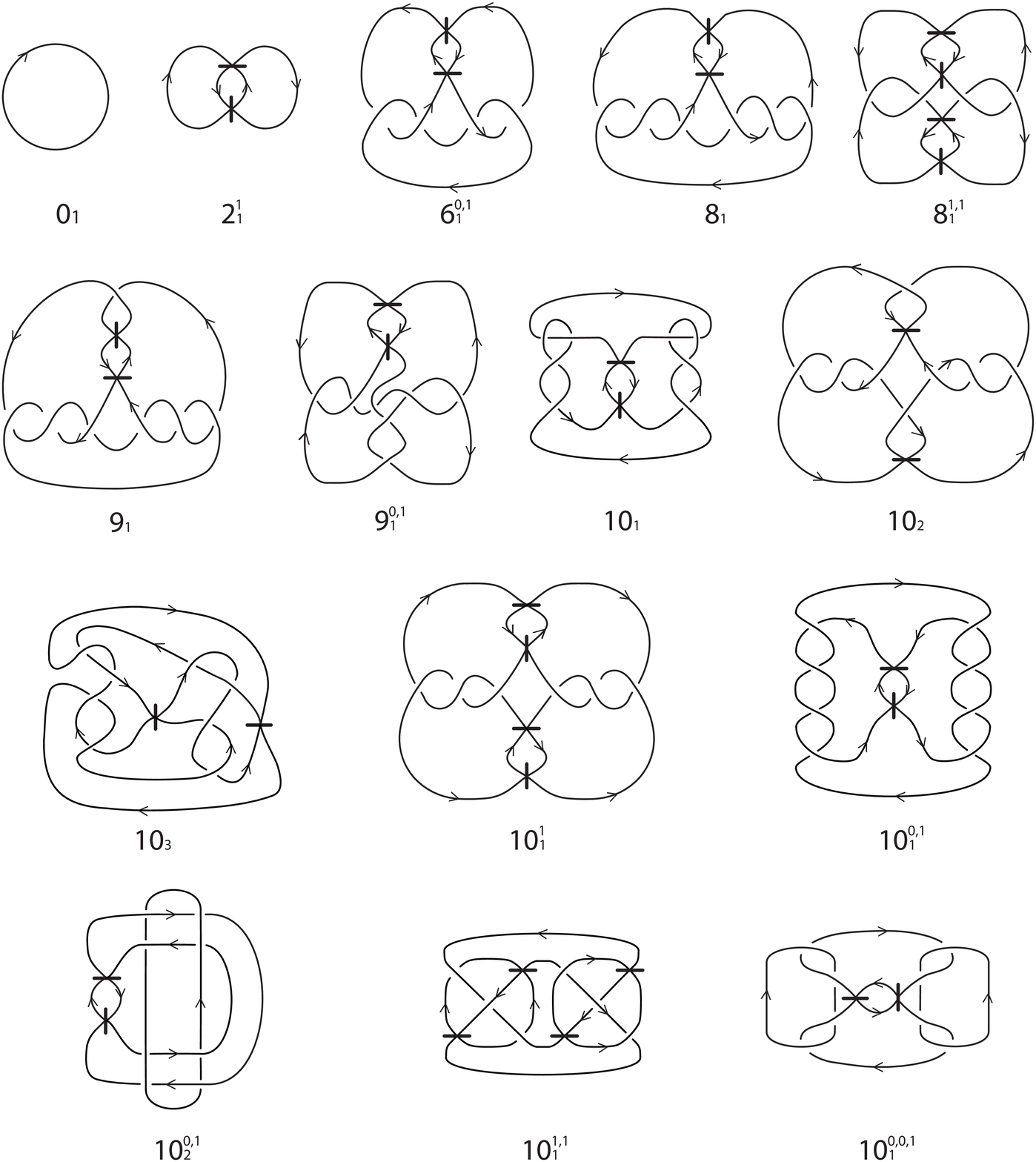} }
\caption{Oriented marked graph diagrams $\Gamma$ with $\chi(\Gamma)\leq 10$}\label{fig-ytab}
\end{center}
\end{figure}

Let ${R_3}$ be the dihedral quandle of order 3 and $\theta$ the 3-cocycle in Example \ref{ex-coc102}. Let $S_4$ be the tetrahedral quandle in Example \ref{ex-qua} and let $\eta=$ 
$$\chi_{0,1,0}\chi_{0,2,1}\chi_{0,2,3}\chi_{0,3,0}\chi_{0,3,1}\chi_{0,3,2}\chi_{1,0,1}\chi_{1,0,3}\chi_{1,2,0}\chi_{1,3,1}\chi_{2,0,3}\chi_{2,1,0}\chi_{2,1,3}\chi_{2,3,2} $$ 
in $Z^3_Q(S_4;\mathbb Z_2),$ where  $\chi_{x,y,z}(a,b,c)=t$ if $(x,y,z)=(a,b,c),$ $\chi_{x,y,z}(a,b,c)=1$ otherwise, and $\mathbb Z_2= \langle t \mid t^2=1 \rangle$. Then $\Phi_\theta(\mathcal L)$ and $\Phi_\eta(\mathcal L)$ are as in the table below.

\begin{center}
\begin{tabular}{|l|l|l|l|l|l|}
\hline
$\mathcal L$ & $\Phi_\theta(\mathcal L)$ & $\Phi_\eta(\mathcal L)$ & $\mathcal L$ & $\Phi_\theta(\mathcal L)$ &  $\Phi_\eta(\mathcal L)$\\
\hline
\hline
$0_{1}$ & 3&4& $10_{2}$ & 3+6u&4\\
$2_{1}^{1}$ & 3&4&$10_{3}$ & 3&4+12t\\
$6_{1}^{0,1}$ & 3&4&$10_{1}^{1}$ & 9&16\\
$8_{1}$ & 9&16&$10_{1}^{0,1}$ & 3&4\\ 
$8_{1}^{1,1}$ & 3&4&$10_{2}^{0,1}$ & 3&4\\
$9_{1}$ & 9&16&$10_{1}^{1,1}$ & 3&4\\
$9_{1}^{0,1}$ & 3&4&$10_{1}^{0,0,1}$ & 9&16\\
$10_{1}$ & 3&4 &&& \\
\hline
\end{tabular}
\vskip 0.5cm
\centerline{Table: $\Phi_\theta(\mathcal L)$ and $\Phi_\eta(\mathcal L)$ with $\chi(\mathcal L)\leq 10$}\label{tb-coc}
\end{center}

\end{example}

In \cite{Yo}, K. Yoshikawa introduced the notion of a marked graph diagram of {\it triangle type}.  It is seen that the quandle cocycle invariant $\Phi_\theta(\mathcal L)$ of an oriented surface-link $\mathcal L$  presented by a marked graph diagram of triangle type 
 is equal to $\#{\rm Col}_X(\mathcal L)$ for  any finite quandle $X$ and any 3-cocycle  $\theta\in Z^3_Q(X;A),$ where $\#{\rm Col}_X(\mathcal L)$ denotes the cardinality of the set ${\rm Col}_X(\mathcal L).$

In \cite{So}, M. Soma gave an enumeration of surface-links presented by marked graph diagrams of {\it square type}; $A_n, B_n, C_n, D_n,E_n, F_n, G_n, H_n$ and $I$ (See \cite{KJL}, \cite[Theorems 1.1 and 1.2]{So}). We remark that  surface-links presented by  marked graph diagrams $A_n$ and $B_n$ are orientable for all $n \geq 2$, and  surface-links presented by  marked graph diagrams $E_n, F_n, G_n, H_n$ and $I$ are also orientable for all odd integers $n \geq 3$.  See Figure \ref{fig-table-sl}.

We observe that for any finite quandle $X$ and $\theta\in Z^3_Q(X;A),$  the quandle cocycle invariant $\Phi_\theta(\mathcal L)$ of an oriented surface-link $\mathcal L$  presented by a marked graph diagram of square type 
 is equal to $\#{\rm Col}_X(\mathcal L)$  except for the surface-link presented by $G_n.$

\begin{figure}[h]
\begin{center}
\resizebox{0.75\textwidth}{!}{%
\includegraphics{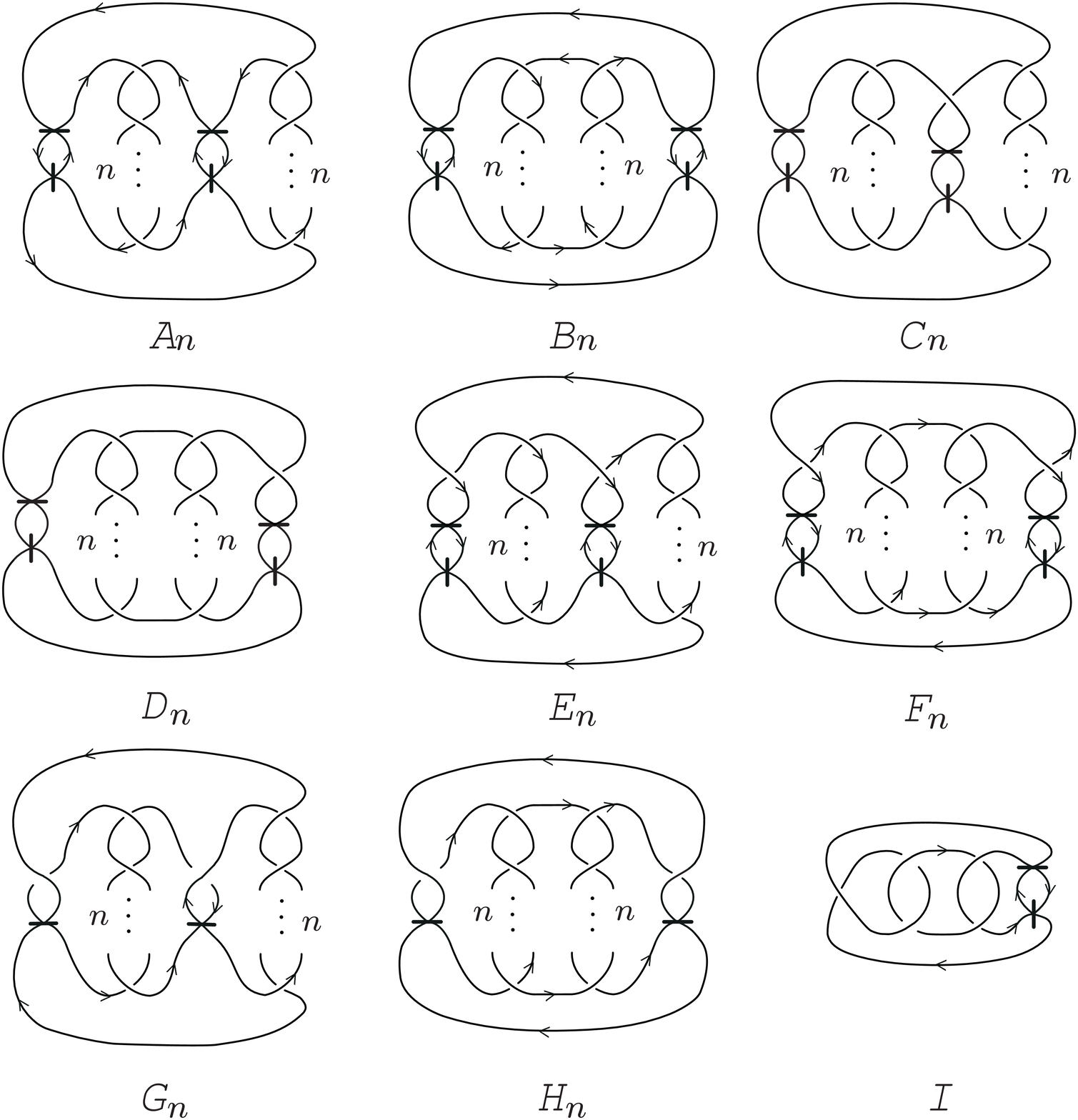} }
\caption{Marked graph diagrams of square type} 
\label{fig-table-sl}
\end{center}
\end{figure}

For an oriented surface-link $\mathcal L$,  we denote the same surface-link as $\mathcal L$ but with the opposite orientations on all the components of $\mathcal L$ by $-\mathcal L$. An oriented surface-link $\mathcal L$ is said to be {\it invertible} if it is equivalent to $-\mathcal L$; otherwise  {\it non-invertible}. The quandle cocycle invariant provides a diagrammatic method of detecting non-invertibility of surface-links (cf. \cite[Section 3]{CKS1}).

\begin{theorem}
 For every integer $k\geq0,$ the oriented surface-links presented by marked graph diagrams $G_{18k+3}$ and $G_{18k+15}$ in Figure \ref{fig-table-sl}   are non-invertible.
\end{theorem}
\begin{proof}  Let $X$ be the dihedral quandle of order 3 and $\theta\in Z^3_Q(X;A)$ the 3-cocycle in Example \ref{ex-coc102}. 
Let $\mathcal L$ be the oriented surface-link presented by the oriented marked graph diagram $G_{n}$ in Figure \ref{fig-table-sl}.
Then $\Phi_\theta(\mathcal L)=3+6u$ and $\Phi_\theta(-\mathcal L)=3+6u^2$ if $n=18k+3$, and $\Phi_\theta(\mathcal L)=3+6u^2$ and $\Phi_\theta(-\mathcal L)=3+6u$ if $n=18k+15$ for any  integer $k\geq0$. This shows that $\mathcal L$ and $-\mathcal L$ are not equivalent for any $k\geq0$ and completes the proof.
\end{proof}  
 
On the other hand, it is shown that for every integer $m\geq1,$ the oriented surface-links presented by marked graph diagrams $F_{2m+1}$ and $H_{2m+1}$ in Figure \ref{fig-table-sl} are all non-invertible \cite[Theorem 7.4]{KJL}.


\section{Shadow quandle cocycle invariants of oriented surface-links}\label{sect-sqcoc}

In this section, we recall shadow  quandle cocycle invariants of oriented surface-links (cf. \cite{CKS}).

   Let $X$ be a quandle and let $\mathcal B$ be a broken surface diagram of an oriented surface-link $\mathcal L$. 
Let $S(\mathcal B)$ be the set of sheets of $\mathcal B$ and $R(\mathcal B)$ be the set of 
the complementary regions of $\mathcal B$ in $\mathbb R^3$.  
   Let $\mathcal{C}: S(\mathcal B)\rightarrow X$ be a coloring of $\mathcal B$.  A {\it shadow coloring} of $\mathcal B$ (extending a given coloring $\mathcal C$) is a map $\tilde{\mathcal{C}}: S(\mathcal B) \cup  R(\mathcal B) \rightarrow X$ satisfying the conditions: 
\begin{itemize}   
\item The restriction of $\tilde{\mathcal C}$ to $S(\mathcal B)$ is a given coloring $\mathcal C$.   

\item   If two adjacent regions $f_1$ and $f_2$  are separated by a sheet $e$ and  the co-orientation of $e$ points from $f_1$ to $f_2$, then $\tilde{\mathcal{C}}(f_1) \ast \tilde{\mathcal{C}}(e) = \tilde{\mathcal{C}}(f_2)$. 

\end{itemize}

Let ${\rm Col}_X^S(\mathcal B)$ be the set of all shadow colorings of $\mathcal B$ by $X$.

Let $\tilde{\mathcal C}$ be a shadow coloring of $\mathcal B$. Let $\tau$ be a triple point  and let $\mathcal R$ be the source region of $\tau$. Let $\theta\in Z^4_Q(X;A).$ Define the {\it shadow (Boltzman) weight} at $\tau$ by $$B_\theta^S(\tau,\tilde{\mathcal C})=\theta(y,x_1,x_2,x_3)^{\epsilon(\tau)},$$ 
where 
 $\epsilon(\tau)$ is the sign of $\tau$, 
$y$ is the color of $\mathcal R$ and $x_1,x_2 $ and $x_3$ are the colors of the bottom, the middle and the top sheets facing $\mathcal R$, respectively.   See Figure \ref{fig-swei} for $\epsilon(\tau)=1$.  
The {\it shadow partition function}  of  $\mathcal B$ (associated to $\theta$) is defined by 
$$\Phi_\theta^s(\mathcal B)=\sum_{\tilde{\mathcal C}\in{\rm Col}_X^S(\mathcal B)}\prod_{\tau\in T(\mathcal B)} B_\theta^S(\tau,\tilde{\mathcal C})\in \mathbb Z[A].$$

 \begin{figure}[ht]
\begin{center}
\resizebox{0.4\textwidth}{!}{%
  \includegraphics{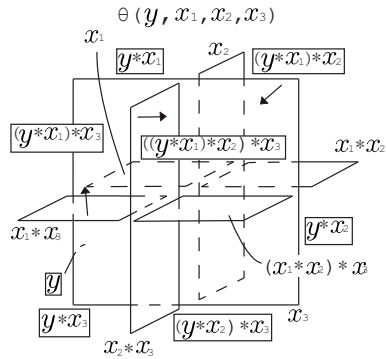} }
\caption{Shadow (Boltzman) weight at $\tau$ with $\epsilon(\tau)=1$}\label{fig-swei}
\end{center}
\end{figure}

\begin{theorem}[\cite{CKS}]\label{thm-scocycleb}
Let $\mathcal B$ be a broken surface diagram of an oriented surface-link $\mathcal L$. The shadow partition function  $\Phi_\theta^s(\mathcal B)$ does not depend on the choice of a broken surface diagram. Thus it is an invariant of $\mathcal L$. 
\end{theorem}

We denote $\Phi_\theta^s(\mathcal B)$ by $\Phi_\theta^s(\mathcal L)$ and call it a {\it shadow quandle cocycle invariant} of $\mathcal L$ associated to $\theta\in Z^4_Q(X;A)$.

There is a generalized version of the shadow quandle cocycle invariant.  

Let $X$ be a quandle.  
The {\it associated group}, $G_X,$ of  $X$  is 
$ \langle x\in X \/ ;  \/ x\ast y=y^{-1}xy \quad (x,y\in X) \rangle$.  
An {\it $X$-set} is a set $Y$ equipped with a right action of the associated group $G_X$.   We denote by $y\ast g$ the image of an element $y\in Y$ by the action $g\in G_X$.  

Let $X$ be a quandle and $Y$ an $X$-set. For each positive integer $n$, let $C_n^R(X)_Y$ be the free abelian group generated by the elements $(y,x_1,\ldots,x_n)$ where $y\in Y$ and $x_1,\ldots,x_n\in X$.  Let $C_0(X)_Y=\mathbb{Z}(Y),$ the free abelian group on $Y$, and let $C_n^R(X)_Y$ be $\{0\}$ for $n<0.$
Define a homomorphism $\partial_n:C_n^R(X)_Y\rightarrow C_{n-1}^R(X)_Y$ by \begin{align*}
\partial_n(y,x_1,x_2,\ldots,x_n)&=\sum^{n}_{i=1} (-1)^i  [(y,x_1,x_2,\ldots,x_{i-1},x_{i+1},\ldots,x_n)\\&-(y\ast x_i,x_1\ast x_i,x_2\ast x_i,\ldots,x_{i-1}\ast x_i,x_{i+1},\ldots,x_n)]\end{align*}
for $n\geq 2$ and $\partial_n=0$ for $n\leq1.$  Then $C_*^R(X)_Y= \{C_n^R(X)_Y,\partial_n\}$ is a chain complex.  
This chain complex is due to R. Fenn, C. Rourke and B. Sanderson (\cite{FeRoSa1, FeRoSa2}).
    Let $D_n^Q(X)_Y$ be the subgroup of $C_n^R(X)_Y$ generated by $(y,x_1,\ldots,x_n)$ with $x_i=x_{i+1}$ for some $i\in\{1,\ldots,n-1\}$ if $n\geq2;$ otherwise let $D_n^Q(X)_Y= \{0\}$.  Then  $C_*^D(X)_Y= \{D_n^Q(X)_Y,\partial_n\}$ is a sub-complex of $C_*^R(X)_Y$. 
Put  $C_n^{Q}(X)_Y =C_n^R(X)_Y/D_n^Q(X)_Y$, and consider the quotient chain complex  
$C_*^{Q}(X)_Y= \{C_n^Q(X)_Y,\partial_n\}$.  
For an abelian group $A$, we define chain and cochain complexes by $C_*^Q(X,A)_Y=C_*^Q(X)_Y\otimes A$ and $C^*_Q(X,A)_Y={\rm Hom}(C_*^Q(X)_Y, A)$.   The homology and cohomology groups are denoted by $H_*^Q(X,A)_Y$ and $H^*_Q(X,A)_Y,$ respectively. For more details, see \cite{Ka2, KO}.

\bigskip

Let $X$ be a quandle, $Y$ an $X$-set and let $\mathcal B$ be a broken surface diagram of an oriented surface-link $\mathcal L$. 
Let $S(\mathcal B)$ be the set of sheets of $\mathcal B$ and $R(\mathcal B)$  the set of 
the complementary regions of $\mathcal B$ in $\mathbb R^3$.  
   Let $\mathcal{C}: S(\mathcal B)\rightarrow X$ be a coloring of $\mathcal B$.  A {\it shadow coloring} of $\mathcal B$ (extending a given coloring $\mathcal C$) by $(X,Y)$ 
   is a map $\tilde{\mathcal{C}}: S(\mathcal B) \cup  R(\mathcal B) \rightarrow X \cup Y $ satisfying the conditions:  
\begin{itemize}   
\item 
$\tilde{\mathcal{C}} (S(\mathcal B)) \subset X$ and $\tilde{\mathcal{C}} (R(\mathcal B)) \subset Y$. 

\item The restriction of $\tilde{\mathcal C}$ to $S(\mathcal B)$ is a given coloring $\mathcal C$.   

\item   If two adjacent regions $f_1$ and $f_2$ are separated by a sheet $e$ 
and the co-orientation of $e$ points from $f_1$ to $f_2$, then $\tilde{\mathcal{C}}(f_1)\ast \tilde{\mathcal{C}}(e) =\tilde{\mathcal{C}}(f_2)$.  

\end{itemize}

We denote by ${\rm Col}_{(X,Y)}^S(\mathcal B)$  the set of all shadow colorings of $\mathcal B$ by $(X,Y)$. 

\begin{proposition}[cf. \cite{CKS}]\label{prop-scol-a}
If $\mathcal B$ and $\mathcal B'$ present equivalent oriented surface-links, then there is a bijection between 
${\rm Col}_X^S(\mathcal B)$  
and ${\rm Col}_X^S(\mathcal B')$, and there is a bijection between 
${\rm Col}_{(X,Y)}^S(\mathcal B)$ 
and ${\rm Col}_{(X,Y)}^S(\mathcal B')$.  
\end{proposition}

Let $\tilde{\mathcal C}$ be a shadow coloring of a broken surface diagram $\mathcal B$ by $(X,Y)$. Let $\tau$ be a triple point  and let $\mathcal R$ be the source region of $\tau$. Let $\theta\in Z^3_Q(X;A)_Y.$ Define the {\it shadow (Boltzman) weight} at $\tau$ by 
$$B_\theta^S(\tau,\tilde{\mathcal C})=\theta(y,x_1,x_2,x_3)^{\epsilon(\tau)},$$ 
where 
$\epsilon(\tau)$ is the sign of $\tau$, 
$y$ is the color of $\mathcal R$ and $x_1,x_2 $ and $x_3$ are the colors of the bottom, the middle and the top sheets facing $\mathcal R$, respectively.   See Figure \ref{fig-swei}.  The {\it shadow partition function} of $\mathcal B$ (associated to $\theta$) is defined by 
$$\Phi_\theta^s(\mathcal B)=\sum_{\tilde{\mathcal C}\in{\rm Col}_{(X,Y)}^S(\mathcal B)}\prod_{\tau\in T(\mathcal B)} B_\theta^S(\tau,\tilde{\mathcal C})\in \mathbb Z[A].$$

\begin{theorem}[cf. \cite{CKS}]\label{thm-scocyclebB}
Let $\mathcal B$ be a broken surface diagram of an oriented surface-link $\mathcal L$. The shadow partition function  $\Phi_\theta^s(\mathcal B)$ does not depend on the choice of a broken surface diagram. Thus it is an invariant of $\mathcal L$. 
\end{theorem}

We denote $\Phi_\theta^s(\mathcal B)$ by $\Phi_\theta^s(\mathcal L)$ and  call it a {\it shadow quandle cocycle invariant} of $\mathcal L$ associated to $\theta\in Z^3_Q(X;A)_Y$. 


\section{How to compute shadow quandle cocycle invariants from marked graph diagrams}\label{sect-sqcocm}

In this section we give a method of computing shadow quandle cocycle invariants from marked graph diagrams.

Let $\Gamma$ be a marked graph diagram of an oriented surface-link $\mathcal L$. 
Let $A(\Gamma)$ be the set of arcs of $\Gamma$ and $R(\Gamma)$ 
the set of complementary regions  of $\Gamma$ in $\mathbb R^2$.  Let $X$ be a quandle and let $Y$ be an $X$-set.  
Let $\mathcal C: A(\Gamma)\rightarrow X$ be a coloring of $\Gamma$ by a quandle $X$.  
A {\it shadow coloring} of $\Gamma$ (extending a given coloring $\mathcal C$) by $X$ (or by $(X,Y)$, resp.) 
is a map $\tilde{\mathcal C}: A(\Gamma) \cup R(\Gamma) \rightarrow X$ 
(or a map $\tilde{\mathcal C}: A(\Gamma) \cup R(\Gamma) \rightarrow X \cup Y$, resp.)  
satisfying the conditions (2) and (3) (or the conditions (1)--(3), resp.):

\begin{itemize}
\item[(1)] $\tilde{\mathcal C} (A(\Gamma)) \subset X$ and $\tilde{\mathcal C} (R(\Gamma)) \subset Y$. 
\item[(2)] The restriction of $\tilde{\mathcal C}$ to $A(\Gamma)$ is a given coloring $\mathcal C$.  
\item[(3)] If two adjacent regions $f_1$ and $f_2$  are separated by an arc $e\in A(\Gamma)$  and  the co-orientation of $e$ points from $f_1$ to $f_2$, then $\tilde{\mathcal{C}}(f_1) \ast \tilde{\mathcal{C}}(e) = \tilde{\mathcal{C}}(f_2)$. 

\end{itemize}

Let ${\rm Col}_X^S(\Gamma)$ (or ${\rm Col}_{(X,Y)}^S(\Gamma)$, resp.) 
denote the set of all shadow colorings of $\Gamma$ by $X$ (or by $(X,Y)$, resp.).

\begin{theorem}\label{thm-scol}
Let $\Gamma$ be a marked graph diagram of an oriented surface-link $\mathcal L$ and $\mathcal B=\mathcal B(\Gamma)$ an associated broken surface diagram of $\Gamma.$  There is a bijection from ${\rm Col}_X^S(\Gamma)$ to ${\rm Col}_X^S(\mathcal B)$, and  a bijection  from ${\rm Col}_{(X,Y)}^S(\Gamma)$ to ${\rm Col}_{(X,Y)}^S(\mathcal B)$.  
\end{theorem}

\begin{proof}
Consider a shadow coloring of $\mathcal B$. The $0$-level cross-section with the colors induced by  the shadow coloring of $\mathcal B$ is a shadow coloring of $\Gamma.$ By the same argument as in \cite{As}, we see that this gives a bijection from ${\rm Col}_X^S(\Gamma)$ to ${\rm Col}_X^S(\mathcal B)$ and a bijection ${\rm Col}_{(X,Y)}^S(\Gamma)$ to ${\rm Col}_{(X,Y)}^S(\mathcal B)$.  
\end{proof}

\begin{corollary}
If $\Gamma$ and $\Gamma'$ present equivalent oriented surface-links, then there is a bijection from ${\rm Col}_X^S(\Gamma)$ to ${\rm Col}_X^S(\Gamma')$, and there is a bijection from ${\rm Col}_{(X,Y)}^S(\Gamma)$ to ${\rm Col}_{(X,Y)}^S(\Gamma')$.  
\end{corollary} 

\begin{proof} 
Let $\mathcal B(\Gamma)$ and $\mathcal B(\Gamma')$ be broken surface diagrams associated to $\Gamma$ and $\Gamma'$, respectively.  
By Proposition~\ref{prop-scol-a} 
and Theorem~\ref{thm-scol}, we see the result. 
\end{proof} 

\bigskip 

Let $\Gamma$ be a marked graph diagram of an oriented surface-link $\mathcal L$ 
and let 
$\Gamma_+=D_1 \rightarrow D_2 \rightarrow\cdots\rightarrow D_r=O$, 
$\Gamma_-=D_1' \rightarrow D_2' \rightarrow\cdots\rightarrow D_s'=O'$, 
$\epsilon_{tm}$ and $\epsilon_{b}$ be as in Section~\ref{sect-qcocm}. 
Let $\tilde{\mathcal{C}}: A(\Gamma) \cup R(\Gamma) \rightarrow X$ or $\tilde{\mathcal{C}}: A(\Gamma) \cup R(\Gamma) \rightarrow X \cup Y$ 
be a shadow coloring of  $\Gamma$. Let $i\in I^3_+$ (or $j\in {I^3_-}$).  Let $R$ be the source region of the crossing $c$ between the top arc and the middle arc in  $D_{i}\cap B_{(i)}$ (or $D_j'\cap B_{(j)}'$).  Let $R'$ be the opposite region of $R$ with respect to the top arc.

Let $x_1$, $x_2$ and $x_3$ be as in Section~\ref{sect-qcocm}. 
There are two cases, the bottom arc intersects with the source region $R$ or not. If not, we consider two cases, $\epsilon_{b}(i)=1$ or $\epsilon_{b}(i)=-1$ (See Figure \ref{fig-strip}).  In the case where the bottom arc hits the source region $R$, the region $R$ is divided by the bottom arc. Let $y$ be the color of the divided region of $R$ such that the co-orientation of the bottom arc points from that region. In the case where the bottom arc does not intersect with the source region $R$ and $\epsilon_{b}(i)=1$, let  $y$ be the element $s=\tilde{\mathcal{C}}(R)$. In the case where the bottom arc does not meet the source region $R$ and $\epsilon_{b}(i)=-1$, let  $y$ be the element $s\ast\overline{x_1}$, where $s=\tilde{\mathcal{C}}(R)$.  For $j\in I^3_-$, let $y$ be the element defined in the same way with $i\in I^3_+.$

\begin{definition}\label{def-sweim}
Let $\theta\in Z^4_Q(X;A)$ be a $4$-cocycle or let $\theta\in Z^3_Q(X;A)_Y$ be a $3$-cocycle. The {\it shadow (Boltzman) weight}  for $i\in I^3_+\amalg I^3_-$ is defined by $$B^S_\theta(i,\tilde{\mathcal{C}})=\theta(y,x_1,x_2,x_3)^{\epsilon_{tm}(i)\epsilon_{b}(i)}.$$
\end{definition}

\newpage
 \begin{figure}[ht]
\begin{center}
\resizebox{0.97\textwidth}{!}{%
  \includegraphics{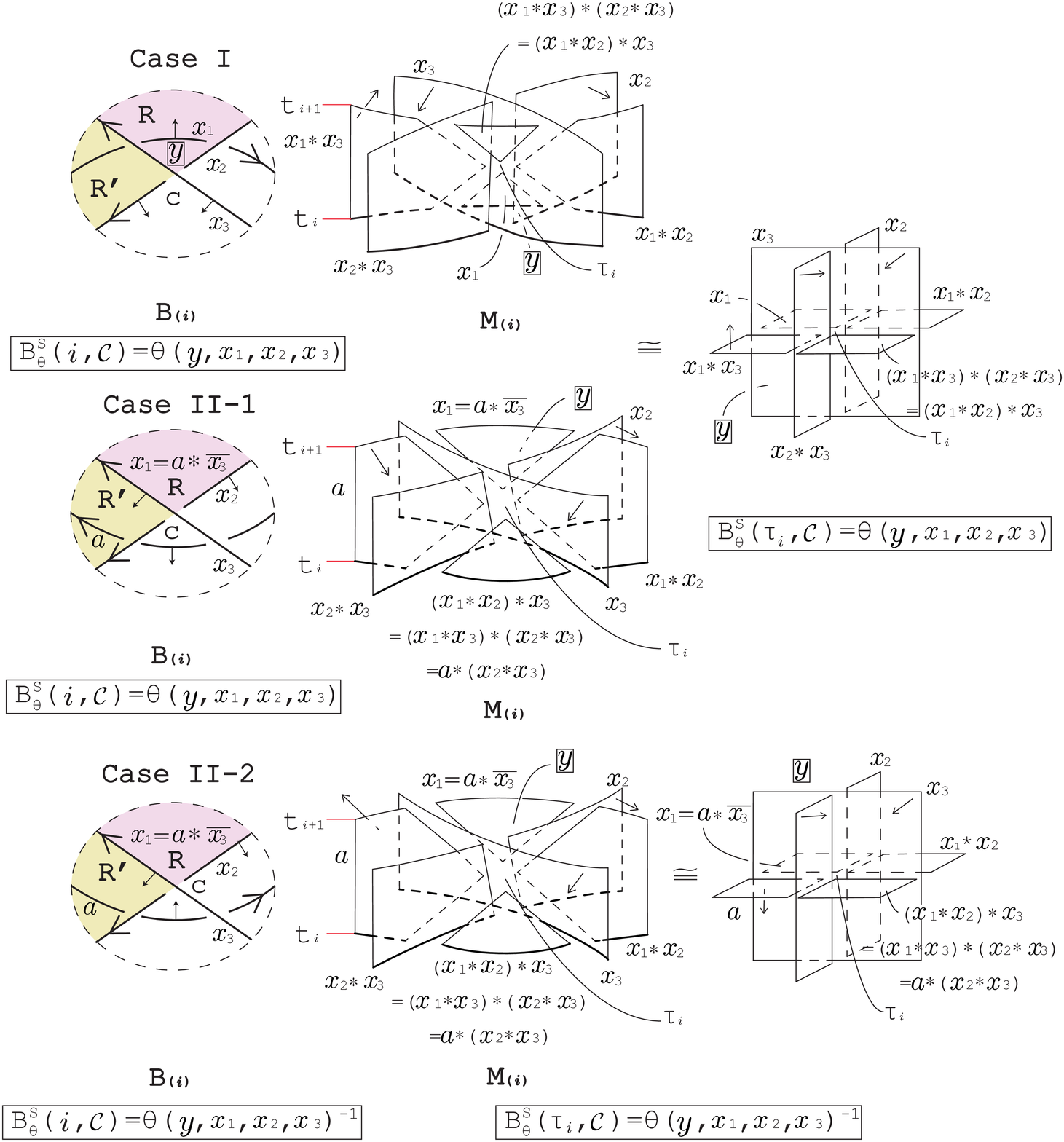} }
\caption{The triple point $\tau_i$}\label{fig-strip}
\end{center}
\end{figure}

\begin{definition}\label{def-scocyclem}
Let $\Gamma$ be a marked graph diagram of an oriented surface-link $\mathcal L.$  The {\it shadow partition function} of $\Gamma$ (associated to  $\theta$)  is defined by  
\begin{align*}
\Phi_\theta^s(\Gamma)= \sum_{\tilde{\mathcal{C}}} \biggl(\prod_{i\in I^3_+ }B^S_\theta(i,\tilde{\mathcal{C}})\prod_{j\in I^3_- }B^S_\theta(j,\tilde{\mathcal{C}})^{-1}\biggr),\end{align*}
where $\tilde{\mathcal{C}}$ runs all shadow colorings of $\Gamma$ by $X$ when $\theta\in Z^4_Q(X;A)$ or  
all shadow colorings of $\Gamma$ by $(X,Y)$ when $\theta\in Z^3_Q(X;A)_Y$.  
\end{definition}

\begin{theorem}\label{thm-scocyclem} 

Let $\mathcal L$ be an oriented surface-link and let $\Gamma$ be a marked graph diagram presenting $\mathcal L$.  For any  $\theta\in Z^4_Q(X;A)$ or  $\theta\in Z^3_Q(X;A)_Y$,  $\Phi_\theta^s(\mathcal L)=\Phi_\theta^s(\Gamma)$.   
\end{theorem}

\begin{proof}
Let $\mathcal B=\mathcal B(\Gamma)$ be a broken surface diagram associated to $\Gamma$. It suffices to show that $\Phi_\theta^s(\mathcal B)=\Phi_\theta^s(\Gamma)$.   

We define  $M_{(i)}$ and $M_{(j)}'$  for any $i\in I^3_+$ and $j\in I^3_-$ as in the proof of Theorem \ref{thm-cocyclem}.  
We have that $T(\mathcal B)=  \{ \tau_i \mid i \in I^3_+ \}  \cup \{ \tau_j' \mid  j \in I^3_- \}$,  
where $\tau_i$ is the triple point in $M_{(i)}$ for $i\in I^3_+$  and $\tau_j'$ is the  triple point in $M_{(j)}'$ for  $j\in I^3_-$.  

Let $\tilde{\mathcal C}$ be a shadow coloring of $\mathcal B$ by $X$ when $\theta\in Z^4_Q(X;A)$ or 
a shadow coloring of $\mathcal B$ by $(X,Y)$ when $\theta\in Z^3_Q(X;A)_Y$, 
and let  $\mathcal C\in{\rm Col}_X(\mathcal B)$ be the restriction of $\tilde{\mathcal C}$ to the set $S(\mathcal B)$. 

We show that 
$B^S_\theta(i,\tilde{\mathcal C})=B^S_\theta(\tau_i,\tilde{\mathcal C})$ for each $i\in I^3_+$. 

The exponents appearing in  $B^S_\theta(i,\tilde{\mathcal C})$ and $B^S_\theta(\tau_i,\tilde{\mathcal C})$ are identical.  The second, the third and the fourth coordinates of $B^S_\theta(\tau,\tilde{\mathcal C})$ are the same as $B_\theta(\tau,\mathcal C)$ for every triple point $\tau.$ Also, the second, third and fourth coordinates of $B^S_\theta(i,\tilde{\mathcal C})$  are the same as $B_\theta(i,\mathcal C)$ for each $i\in I^3_+$. Combining these facts, the second, the third and the fourth coordinates of $B^S_\theta(i,\tilde{\mathcal C})$ are the same as those of $B^S_\theta(\tau_i,\tilde{\mathcal C})$ for any $i\in I^3_+$.

It remains to show that the first coordinate of 
$B^S_\theta(\tau_i,\tilde{\mathcal C})$ is identical with that of 
$B^S_\theta(i,\tilde{\mathcal C})$ for any $i\in I^3_+$. The first coordinate of 
$B^S_\theta(\tau_i,\tilde{\mathcal C})$ is the color of the source region $\mathcal R$ of the triple point $\tau_i.$


{\bf Case I} : The bottom arc intersects with the source region $R$ of the crossing between the top and middle arc  in $B_{(i)}.$

The top (or the middle, resp.) sheet corresponds to the top (or the middle, resp.) arc times $[t_i,t_{i+1}].$ Also, the quadrant between the top and middle sheets with the co-orientations outward is divided into two ($3$-dimensional) regions 
by a bottom sheet whose color is same as that of the bottom arc in $R$. Therefore, the first coordinate of $B^S_\theta(\tau_i,\tilde{\mathcal C})$ is $\tilde{\mathcal C}(R),$ where $R$ is the source region of the crossing $c$ between the top and middle arcs in $B_{(i)}.$ Therefore $B^S_\theta(i,\tilde{\mathcal C})=B^S_\theta(\tau_i,\tilde{\mathcal C})$ for all $i\in I^3_+.$
 

{\bf Case II} : 
The bottom arc does not hit the source region $R$ of the crossing between the top and middle arc in $B_{(i)}.$

Let $\tilde{\mathcal C}(R)=s.$ The quadrant corresponding to $R\times[t_i,t_{i+1}]$ is divided into two ($3$-dimensional) regions 
by a bottom sheet whose color is $x_1=a\ast \overline{x_3}.$ If $\epsilon_b(i)=1$, then the co-orientation of the bottom sheet in that quadrant is from the region which has a color $s$ (see the case II-1 in Figure \ref{fig-strip}). Therefore the color $y$ of the source region $\mathcal R$ of the triple point $\tau_i$ is $s.$ Otherwise, the co-orientation of the bottom sheet in that quadrant points to the region whose color is $s$ (see the case II-2 in Figure \ref{fig-strip}). In addition, the color of the bottom sheet in that quadrant is $x_1.$ Thus the color $y$ of $\mathcal R$ is $s\ast\overline{x_1}.$ Therefore $B^S_\theta(i,\tilde{\mathcal C})=B^S_\theta(\tau_i,\tilde{\mathcal C})$ for all $i\in I^3_+.$

\smallskip
For $j\in I^3_-,$ it is similarly seen that $B^S_\theta(j,\tilde{\mathcal C})=B^S_\theta(\tau_j',\tilde{\mathcal C})^{-1}$.  

Hence  we have  $\Phi_\theta^s(\mathcal B)=\Phi_\theta^s(\Gamma)$ for all $j\in I^3_-.$     
\end{proof}


\section{Symmetric quandle cocycle invariants of unoriented surface-links}\label{sect-syqcoc}

This section is devoted to recalling symmetric quandle cocycle invariants of unoriented surface-links (cf.  \cite{Ka2, KO}). 

Let $X$ be a quandle. A map $\rho:X\rightarrow X$ is a {\it good involution} if it is an involution (i.e., $\rho\circ\rho={\rm id}$) such that $\rho(x\ast y)=\rho(x)\ast y$ and $x\ast\rho(y)=x\ast\bar y$ for any $x,y\in X.$ Such a pair $(X,\rho)$ is called a {\it quandle with a good involution} or a {\it symmetric quandle.}

\begin{example}(\cite{Ka2, KO})
Let $G$ be a group. The inversion, inv$(G):G\rightarrow G;g\mapsto g^{-1}$, is a good involution of conj$(G)$. We call $({\rm conj}(G),{\rm inv}(G))$ the {\it conjugation symmetric quandle.}
\end{example}

The {\it associated group}, $G_{(X,\rho)},$ of a symmetric quandle $(X,\rho)$ is  
$G_{(X,\rho)} = \langle x\in X \/ ; \/ x\ast y=y^{-1}xy~(x,y\in X),~\rho(x)=x^{-1}~ (x\in X) \rangle$.  
An {\it $(X,\rho)$-set} is a set $Y$ equipped with a right action of the associated group $G_{(X,\rho)}.$ We denote by $y\ast g$ the image of an element $y\in Y$ by the action $g\in G_{(X,\rho)}.$

Let $(X,\rho)$ be a symmetric quandle and $Y$ an $(X,\rho)$-set. 
Let $C_*^R(X)_Y=\{C_n^R(X)_Y,\partial_n\}$ be the chain complex of $X$ with $Y$, and 
$C_*^D(X)_Y=\{D_n^Q(X)_Y,\partial_n\}$ be the sub-complex of  $C_*^R(X)_Y$ 
as in Section~\ref{sect-sqcoc}. 

Let $D_n^\rho(X)_Y$ be the subgroup of $C_n^R(X)_Y$ generated by $$(y,x_1,\ldots,x_n)+(y\ast x_j,x_1\ast x_j,\ldots, x_{j-1}\ast x_j,\rho(x_j),x_{j+1}, \cdots,x_n)$$ for $j\in{1,\ldots,n}$ if $n\geq2;$ otherwise let $D_n^\rho(X)_Y= \{0\}$.   

Define $C_n^{Q,\rho}(X)_Y$ to be $C_n^R(X)_Y/(D_n^Q(X)_Y+D_n^\rho(X)_Y)$, and we have the 
quotient complex $C_*^{Q,\rho}(X)_Y = \{ C_n^{Q,\rho}(X)_Y, \partial_n  \}$. 
For an abelian group $A$, we define  chain and cochain complexes by $C_*^{Q,\rho}(X,A)_Y=C_*^{Q,\rho}(X)_Y\otimes A$ and $C^*_{Q,\rho}(X,A)_Y={\rm Hom}(C_*^{Q,\rho}(X)_Y, A)$, respectively. The homology and cohomology groups are denoted by 
$H_*^{Q,\rho}(X,A)_Y$ and $H^*_{Q,\rho}(X,A)_Y,$ respectively. For  details, see \cite{Ka2, KO}.

Let $\mathcal B$ be an unoriented broken surface diagram.   
When we divide over-sheets at the double curves, we call the sheets of the result {\it semi-sheets} of $\mathcal B$.   
Each semi-sheet is a compact orientable surface in $\mathbb R^3$ (cf. \cite{Ka1}).

Consider an assignment of normal orientation and an element of $X$ to each semi-sheet of $\mathcal B$. 
A {\it basic inversion} is an operation which reverses the normal orientation of a semi-sheet and changes the element $x$ assigned to the semi-sheet by $\rho(x).$ See Figure \ref{fig-bi}. 

 \begin{figure}[ht]
\begin{center}
\resizebox{0.7\textwidth}{!}{%
  \includegraphics{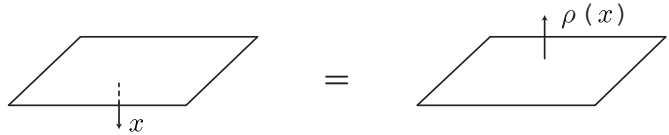} }
\caption{A basic inversion}\label{fig-bi}
\end{center}
\end{figure}

We would rather use the terminology `normal orientation'  than `co-orientation' when $\mathcal B$ is an unoriented broken surface diagram.

An {\it $(X,\rho)$-coloring} of $\mathcal B$ is the equivalence class of an assignment of a normal orientation and an element of $X$ to each semi-sheet of $\mathcal B$ satisfying the coloring condition below. Here the equivalence relation is generated by basic inversions. 
\begin{itemize}
\item[$\bullet$] 
By basic inversions, assume the normal orientations of semi-sheets around a double point curve to be  
as in Figure \ref{fig-colub}.  Then $x_1 \ast x_3 = x_2$ and $x_3 = x_4$.  
\end{itemize}

 \begin{figure}[ht]
\begin{center}
\resizebox{0.27\textwidth}{!}{%
  \includegraphics{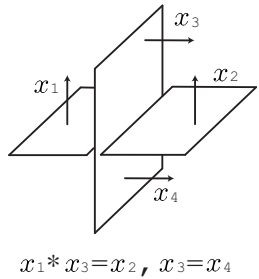} }
\caption{Coloring conditions}\label{fig-colub}
\end{center}
\end{figure}

Let $Y$ be an $(X,\rho)$-set. An {\it $(X,\rho)_Y$-coloring} of $\mathcal B$ is an $(X,\rho)$-coloring of $\mathcal B$ with an assignment of an element of $Y$ to each complementary region of $\mathcal B$ satisfying the following condition.  
\begin{itemize}
\item[$\bullet$] Suppose that adjacent regions $f_1$ and $f_2$ separated by a semi-sheet  $e$ are labeled by $y_1$ and $y_2$. If the semi-sheet $e$ is labeled by $x$ and the normal orientation of $e$ points from $f_1$ to $f_2$, then $y_1\ast x=y_2.$ 
\end{itemize}

%

\begin{proposition}[\cite{Ka2, KO}] 
Let $(X,\rho)$ be a symmetric quandle and $Y$ an $(X,\rho)$-set. If two broken surface diagrams present equivalent unoriented surface-links, then there is a bijection between the sets of $(X,\rho)$-colorings of the broken surface diagrams, and there is a bijection between the sets of $(X,\rho)_Y$-colorings of them.
\end{proposition}

Let $\mathcal B$ be an unoriented broken surface diagram. Fix an $(X,\rho)_Y$-coloring of $\mathcal B$, say $\tilde{\mathcal C}$. For a triple point $\tau$ of $\mathcal B$, there are eight complementary regions of $\mathcal B$ around $\tau$ (Some of them may be the same). Choose one of them, say $f$, which we call a {\it specified region} for $\tau$, and let $y$ be the label of $f$.

Let $e_1,$ $e_2$ and $e_3$ be the bottom semi-sheet, the middle semi-sheet and the top semi-sheet at $\tau$, respectively, which face the region $f$. By basic inversions, we assume that the normal orientations $n_1$, $n_2$ and $n_3$ of them point from $f$ to the opposite regions. Let $x_1,$ $x_2$ and $x_3$ be the labels of them, respectively. The {\it sign} of $\tau$ with respect to the region $f$ is $+1$ (or $-1$) if the triple of normal orientations $(n_3,n_2,n_1)$ does (or does not) match the orientation of $\mathbb R^3$. Let $\theta\in Z^3_{Q,\rho}(X,A)_Y.$ The {\it symmetric (Boltzman) weight} $B_\theta^{Sym}(\tau,\tilde{\mathcal C})$ of $\tau$ is defined to be $$B_\theta^{Sym}(\tau,\tilde{\mathcal C})=\theta(y,x_1,x_2,x_3)^{\epsilon(\tau)},$$ where $\epsilon(\tau)$ is the sign of $\tau$. See Figure \ref{fig-wei}.

 \begin{figure}[ht]
\begin{center}
\resizebox{0.65\textwidth}{!}{%
  \includegraphics{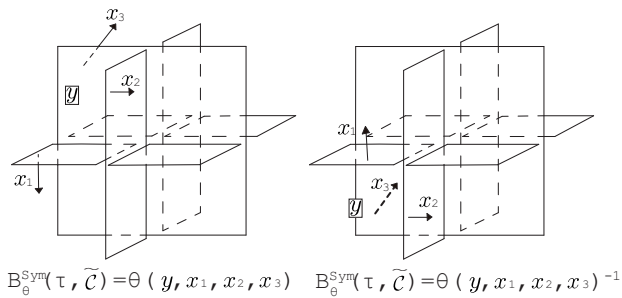} }
\caption{Symmetric Boltzman weights}\label{fig-wei}
\end{center}
\end{figure}

The {\it symmetric partition function} of $\mathcal B$ (associated to  $\theta$) is defined   by 
$$\Phi_\theta^{Sym}(\mathcal B)=\sum_{\tilde{\mathcal C}}\prod_{\tau\in T(\mathcal B)} B_\theta^{Sym}(\tau,\tilde{\mathcal C})\in \mathbb Z[A],$$
where the sum is taken over all possible $(X,\rho)_Y$-colorings $\tilde{\mathcal C}$ of $\mathcal B$. (The value of $B_\theta^{Sym}(\tau,\tilde{\mathcal C})$ is in the coefficient group $A$ written multiplicatively).

\begin{theorem}[\cite{Ka2, KO}]\label{thm-sycocycle}
Let $\mathcal B$ be a broken surface diagram of an unoriented surface-link $\mathcal L$. The symmetric partition function $\Phi_\theta^{Sym}(\mathcal B)$ is an invariant of the unoriented surface-link $\mathcal L$. 
\end{theorem}

We denote $\Phi_\theta^{Sym}(\mathcal B)$  by $\Phi_\theta^{Sym}(\mathcal L)$ and call it the {\it symmetric quandle cocycle invariant} of $\mathcal L$ associated to $\theta$. 


\section{How to compute symmetric quandle cocycle invariants from marked graph diagrams}\label{sect-syqcocm}

Let $\Gamma$ be a marked graph diagram of an unoriented surface-link $\mathcal L$ and let $(X, \rho)$ be a symmetric quandle.

A {\it semi-arc} of $\Gamma$ is a connected component of $\Gamma \setminus (C(\Gamma) \cup V(\Gamma))$, where $C(\Gamma)$ is the set of crossings and $V(\Gamma)$ is the set of marked vertices of $\Gamma$.  

A {\it basic inversion} is an operation which reverses the normal orientation of a semi-arc and changes the element $x$ assigned to the semi-arc by $\rho(x).$ See Figure \ref{fig-bimgd}. 

 \begin{figure}[ht]
\begin{center}
\resizebox{0.5\textwidth}{!}{%
  \includegraphics{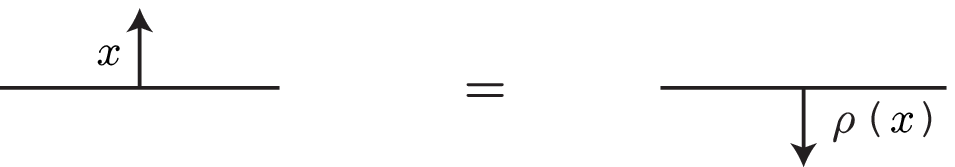} }
\caption{A basic inversion}\label{fig-bimgd}
\end{center}
\end{figure}

We say that an assignment of a normal orientation and an element of $X$ to each semi-arc of $\Gamma$ satisfies the {\it coloring conditions} if it satisfies the following conditions. 
\begin{itemize}

\item For each marked vertex, 
using basic inversions, we assume that normal orientations of semi-arcs are as in Figure~\ref{fig-colmgd}.  
Then $x_1=x_2$.

\item For each crossing, using basic inversions, we assume that normal orientations of semi-arcs are as in Figure~\ref{fig-colmgd}.  Then $x_1\ast x_3=x_2$ and $x_3 =x_4$.  
\end{itemize}

 \begin{figure}[ht]
\begin{center}
\resizebox{0.5\textwidth}{!}{%
  \includegraphics{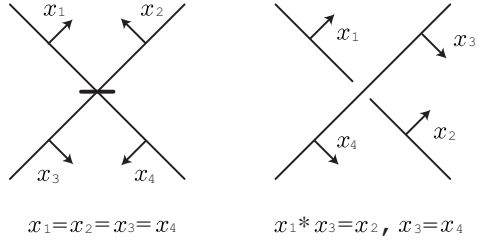} }
\caption{Coloring conditions}\label{fig-colmgd}
\end{center}
\end{figure}


An {\it $(X,\rho)$-coloring} of $\Gamma$ is the equivalence class of an assignment of a normal orientation and an element of $X$ to each semi-arc of $\Gamma$ satisfying the coloring conditions. Here the equivalence relation is generated by basic inversions.

Let $Y$ be an $(X,\rho)$-set. An $(X,\rho)_Y$-coloring of $\Gamma$ is an $(X,\rho)$-coloring with an assignment of an element of $Y$ to each complementary region of $\Gamma$ satisfying the following condition. 

\begin{itemize}
\item[$\bullet$] Suppose that two adjacent regions $f_1$ and $f_2$  separated by a semi-arc $e$ are labeled by $y_1$ and $y_2$. If  the semi-arc $e$ is labeled by $x$ and  the normal orientation of $e$ points from $f_1$ to $f_2$, then $y_1\ast x=y_2.$
\end{itemize}


\begin{theorem}
Let $(X,\rho)$ be a symmetric quandle and let $Y$ be an $(X,\rho)$-set. Let $\Gamma$ be an admissible marked graph diagram, and let $\mathcal B = 
\mathcal B(\Gamma)$ be a broken surface diagram associated with $\Gamma$.  There is a bijection from the set of $(X, \rho)_Y$-colorings of $\mathcal B$ to that of $\Gamma$.  
\end{theorem}

\begin{proof}
By the same argument as in the proof of Theorem~\ref{thm-scol}, we see the result. 
\end{proof}

Let $\Gamma$ be an admissible marked graph diagram. Fix an $(X,\rho)_Y$-coloring of $\Gamma$, say $\tilde{\mathcal C}$. Then both resolutions $\Gamma_+$ and $\Gamma_-$ have induced colorings. 

 \begin{figure}[ht]
\begin{center}
\resizebox{0.7\textwidth}{!}{%
  \includegraphics{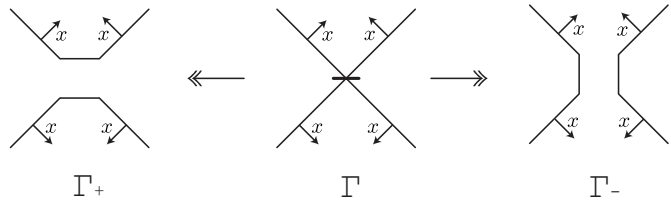} }
\caption{Induced colorings on $\Gamma_+$ and $\Gamma_-$}\label{fig-res}
\end{center}
\end{figure}

Let $\Gamma_+=D_1\rightarrow D_2\rightarrow\cdots\rightarrow D_r=O$ and $\Gamma_-=D_1' \rightarrow D_2'\rightarrow\cdots\rightarrow D_s'=O'$  be sequences of link diagrams as before. Let $i\in I^3_+\amalg I^3_-$ and $f$ the complementary region of $D_i$ in 
$B_{(i)}$ or $B_{(i)}'$ such that $f$ does not intersect with the boundary $\partial B_{(i)}$ or $\partial B_{(i)}'$, respectively. Let $e_1,$ $e_2$ and $e_3$ be the bottom, the middle and the top semi-arcs  facing the region $f$, respectively.  By basic inversions, we assume that the normal orientations $n_1$, $n_2$ and $n_3$ of them point outwards. Let $x_1,$ $x_2$ and $x_3$ be the labels of them, respectively.  Define $\epsilon_{tm}(i)=1$ if $( n_3,n_2 )$ matches with the given (right-handed) orientation of $\mathbb R^2$ and $-1$ otherwise. For a given 3-cocycle $\theta\in Z^3_{Q,\rho}(X,A)_Y,$ we define the {\it symmetric (Boltzman) weight}  at $i$  to be $$B_\theta^{Sym}(i,\tilde{\mathcal C})=\theta(y,x_1,x_2,x_3)^{\epsilon_{tm}(i)}.$$ 

 \begin{figure}[ht]
\begin{center}
\resizebox{0.6\textwidth}{!}{%
  \includegraphics{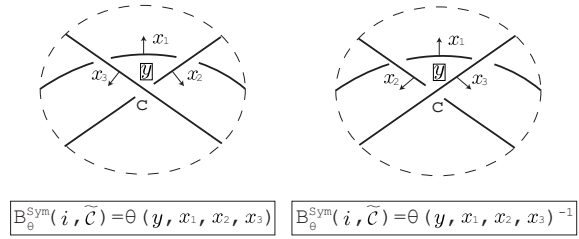} }
\caption{The symmetric (Boltzman) weight at $i\in I^3_+\amalg I^3_-$}\label{fig-weimgd1}
\end{center}
\end{figure}



For a marked graph diagram $\Gamma$ and an $(X,\rho)_Y$-coloring $\tilde{\mathcal C}$, we define the {\it  symmetric partition function}  by $$\Phi_\theta^{Sym}(\Gamma)=\sum_{\tilde{\mathcal C}}\biggl(\prod_{i\in I^3_+ }B^{Sym}_\theta(i,\tilde{\mathcal C})\prod_{j\in I^3_- }B^{Sym}_\theta(j,\tilde{\mathcal C})^{-1}\biggr),$$
where $\tilde{\mathcal C}$ runs over all $(X,\rho)_Y$-colorings of $\Gamma$.

\begin{theorem} \label{thm-cocun}
Let $\mathcal L$ be an unoriented surface-link and let $\Gamma$ be a marked graph diagram presenting  $\mathcal L$. For any  3-cocycle $\theta\in Z^3_{Q,\rho}(X,A)_Y,$ the symmetric partition functions $\Phi_\theta^{Sym}(\Gamma)$ is equal to $\Phi_\theta^{Sym}(\mathcal L)$.
\end{theorem}

\begin{proof} The proof of this theorem is similar to that of Theorem~\ref{thm-scocyclem}. Let $\mathcal B=\mathcal{B}(\Gamma)$ be the broken surface diagram associated to $\Gamma$, and let $\tilde{\mathcal C}\in{\rm Col}_\theta^{Sym}(\mathcal B)$ be an $(X,\rho)_Y$-coloring of $\mathcal B$.  We denote by the same symbol $\tilde{\mathcal C}$ for the corresponding $(X,\rho)_Y$-coloring of $\Gamma$.  
As the oriented case, the set of triple points is   $T(\mathcal B)=  \{ \tau_i \mid 
i \in I^3_+ \}  \cup \{ \tau_j' \mid j \in {I^3_-} \}$,  
where $\tau_i$ is the  triple point in $M_{(i)}$ for $ i\in I^3_+$ and 
$\tau_j'$ is the  triple point in $M_{(j)}'$  for $j\in{I^3_-}.$
     Let $i\in I^3_+$. Since we choose the normal orientation of the bottom arc such that $\epsilon_b(i)=1$, we have $\epsilon(\tau_i)=\epsilon_{tm}(i)$.   Thus $B_\theta^{Sym}(i,\tilde{\mathcal C})=B_\theta^{Sym}(\tau_i,\tilde{\mathcal C})$.   
     Similarly, for $j\in I^3_-,$  we have $B_\theta^{Sym}(j,\tilde{\mathcal C})=B_\theta^{Sym}(\tau_j',\tilde{\mathcal C})^{-1}$.  
Hence  we have 
$\Phi_\theta^{Sym}(\Gamma) =\Phi_\theta^{Sym}(\mathcal B)$.

\begin{figure}[ht]
\begin{center}
\resizebox{1\textwidth}{!}{%
  \includegraphics{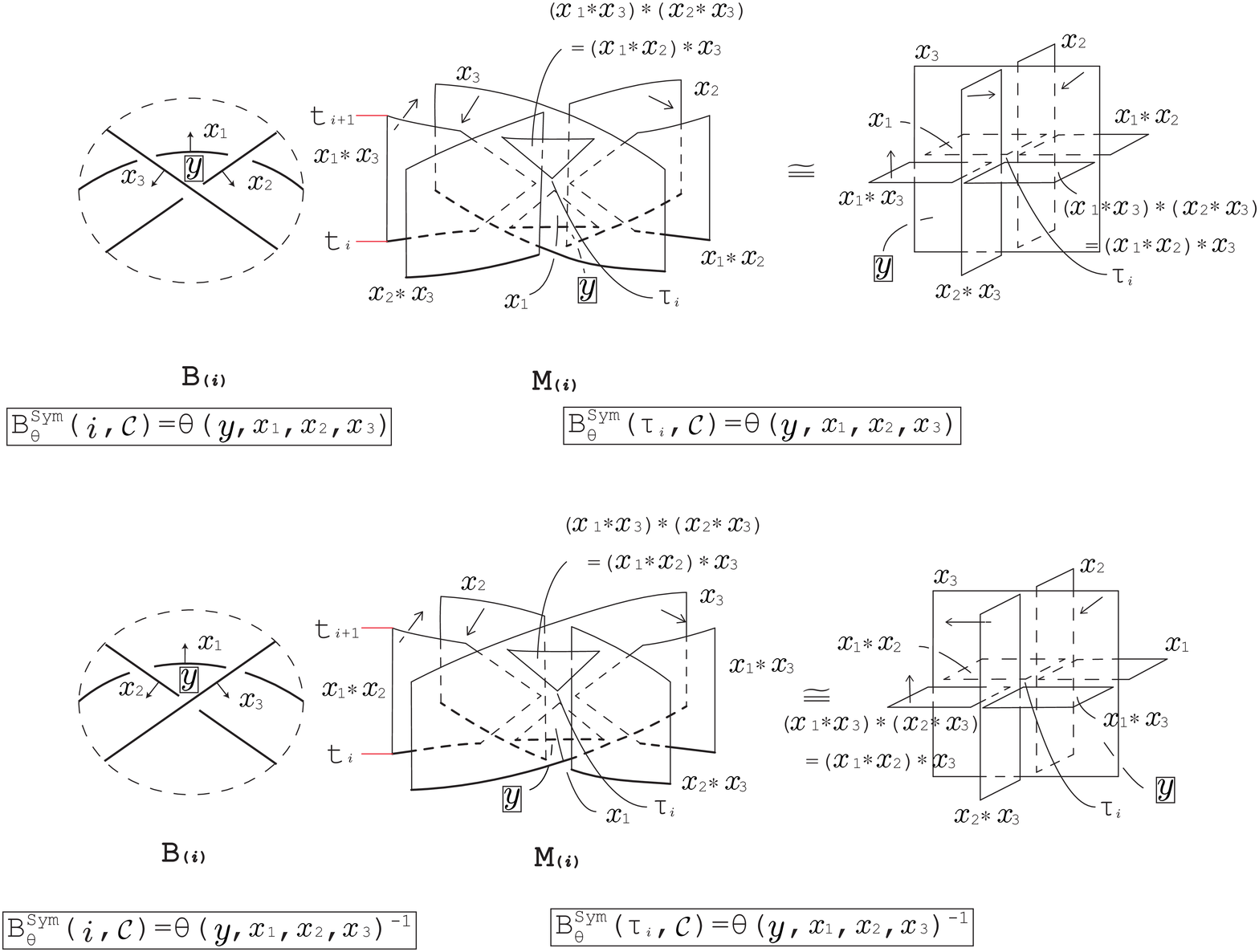} }
\caption{A triple point}\label{fig-tripmgd}
\end{center}
\end{figure}

\end{proof}

\begin{example}\label{ex-81coc}

Let $\Gamma$ be the unorientable marked graph diagram  in Figure \ref{fig-81} representing two component $\mathbb R {\rm P}^2$-link $\mathcal L$ ($\Gamma$ is a marked graph diagram $8_1^{-1,-1}$ in Yoshikawa table \cite{Yo}).

        \begin{figure}[h]
\begin{center}
\resizebox{0.28\textwidth}{!}{%
  \includegraphics{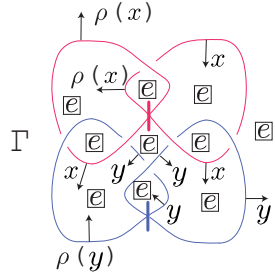} }
\caption{A symmetric coloring for $\Gamma$}\label{fig-81}
\end{center}
\end{figure}
 
   Let $X$ be the dihedral quandle of order 4, in which we rename the elements $0,1,2,3$ by $e_1,e_2,e_1',e_2',$ respectively. Let $\rho:X\rightarrow X$ be the antipodal map, i.e., $\rho(e_i)=e_i'~(i=1,2).$ Let $Y=\{e\},$ which is an $(X,\rho)$-set. Let  \begin{align*}\theta=&{\chi_{e,e_1,e_2,e_1}}\chi_{e,e_1',e_2',e_1}{\chi_{e,e_1',e_2,e_1'}}\chi_{e,e_1,e_2',e_1'}\\
 &{\chi_{e,e_1',e_2,e_1}}^{-1}{\chi_{e,e_1,e_2',e_1}}^{-1}{\chi_{e,e_1,e_2,e_1'}}^{-1}{\chi_{e,e_1',e_2',e_1'}}^{-1}\\
 &{\chi_{e,e_2,e_1,e_2}}^{-1}{\chi_{e,e_2',e_1',e_2}}^{-1}{\chi_{e,e_2',e_1,e_2'}}^{-1}{\chi_{e,e_2,e_1',e_2'}}^{-1}\\
 &\chi_{e,e_2',e_1,e_2}\chi_{e,e_2,e_1',e_2}
\chi_{e,e_2,e_1,e_2'}\chi_{e,e_2',e_1',e_2'}\in Z^3_Q(X;\mathbb Z),
\end{align*}
      where $\chi_{x,y,z,w}(a,b,c,d)=t$ if $(x,y,z,w)=(a,b,c,d),$ $\chi_{x,y,z,w}(a,b,c,d)=1$ otherwise, and $\mathbb Z= \langle t \rangle$ is the infinite cyclic group (cf. \cite[Example 9.3]{KO}).

          Consider sequences of link  diagrams from the positive and negative resolutions to trivial link diagrams as in Figures~\ref{fig-81+} and \ref{fig-81-}, respectively. From those figures, we get $I^3_+=\{2,3,4,6\}$ and ${I^3_-}=\phi.$ The symmetric (Boltzman) weights are 
$B^{Sym}_\theta(2,\mathcal{C})=\theta(e,y,\rho(y),\rho(x))=\theta(e,y,y,\rho(x))^{-1}=1,$  
$B^{Sym}_\theta(3,\mathcal{C})=\theta(e,y,\rho(x),\rho(x))=1,$  
$B^{Sym}_\theta(4,\mathcal{C})=\theta(e,x,y,\rho(x)),$ 
$B^{Sym}_\theta(6,\mathcal{C})\\= \theta(e,\rho(y),x,y)^{-1}$ for $(x,y)\in E,$ where $E=\{(e_1,e_2),(e_1,e_2'),(e_1',e_2),(e_1',e_2'),\\(e_2,e_1),(e_2,e_1'),(e_2',e_1),(e_2',e_1')\}.$
Therefore 
\begin{align*}\Phi_\theta(\mathcal L)&=\sum_{\tilde{\mathcal C}}\biggl(\prod_{i\in I^3_+ }B^{Sym}_\theta(i,\tilde{\mathcal C})\prod_{j\in I^3_- }B^{Sym}_\theta(j,\tilde{\mathcal C})^{-1}\biggr)\\
&=\theta(e,x,y,\rho(x))\theta(e,\rho(y),x,y)^{-1}\\
&=4+2t^2+2t^{-2}.
\end{align*}

 \begin{figure}
\begin{center}
\resizebox{0.88\textwidth}{!}{%
  \includegraphics{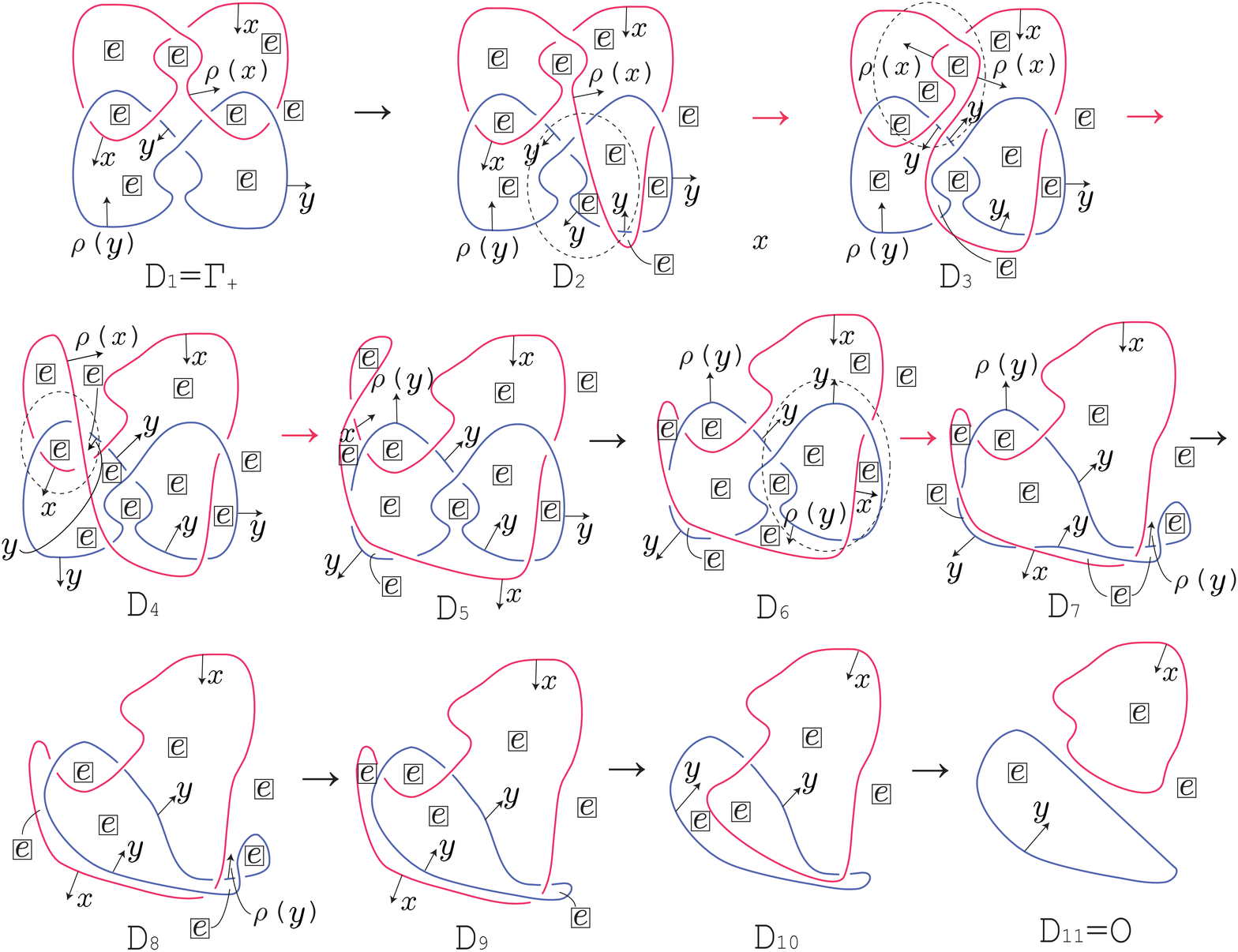} }
\caption{A sequence of link diagrams for $\Gamma_+$}\label{fig-81+}
\end{center}
\begin{center}
\resizebox{0.79\textwidth}{!}{%
  \includegraphics{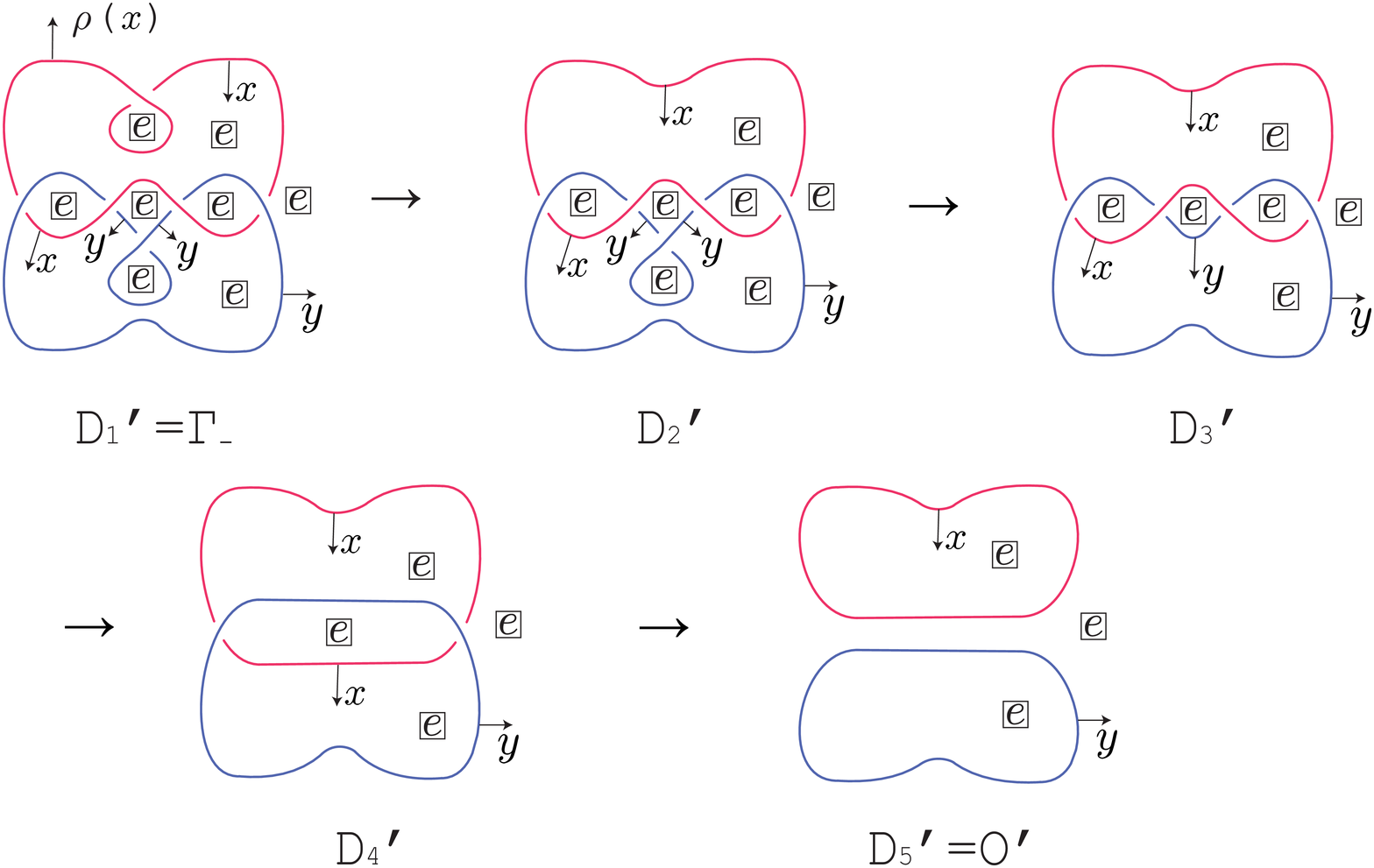} }
\caption{A sequence of link diagrams for $\Gamma_-$}\label{fig-81-}
\end{center}
\end{figure}

\end{example}

{\bf Acknowledgements.}
The first author was supported by JSPS KAKENHI Grant Number 26287013.
The third author was supported by Basic Science Research Program through the National Research Foundation of Korea(NRF) funded by the Ministry of Education, Science and Technology (2013R1A1A2012446).


\end{document}